\documentclass[11pt]{article}
\usepackage{amsmath,amssymb,epsfig,amsthm}

\usepackage{geometry}
\usepackage{floatpag}
\usepackage{url}
\usepackage{multirow}

\usepackage{layout}
\usepackage{color}
\usepackage{gauss}

\usepackage{graphicx}
\usepackage{epsfig}
\usepackage[subrefformat=parens,labelformat=parens]{subfig}
\usepackage{grffile}

\usepackage[algo2e,vlined,longend,linesnumbered,ruled]{algorithm2e}

\usepackage{color}
\usepackage{tikz}
\usetikzlibrary{backgrounds,decorations.pathmorphing,patterns,external,calc}
\usepackage[normalem]{ulem}

\newcommand{\change}[1]{{\color{black}#1}}

\newcommand{\dt}{\,\mathsf{d}t}

\def\VR{\kern-\arraycolsep\strut\vrule &\kern-\arraycolsep}
\def\vr{\kern-\arraycolsep & \kern-\arraycolsep}

\newcommand{\normal}{{\mathcal N}}

\newcommand{\R}{{\mathbb R}}

\def\trace{\mathop{\rm trace}\nolimits}

\newcommand{\lam}{\lambda}
\newcommand{\Lam}{\Lambda}

\newcommand{\mb}[1]{\left[\begin{array}{#1}}
\newcommand{\me}{\end{array}\right]}
\newcommand{\smb}{\left[\begin{smallmatrix}}
\newcommand{\sme}{\end{smallmatrix}\right]}

\newcommand{\EE}{\mathbb{E}}
\newcommand{\PP}[1]{\mathbb{P}\left\{#1\right\}}
\newcommand{\CDF}{F}
\newcommand{\CHAR}{\varphi}

\numberwithin{equation}{section}
\newtheorem{theorem}{Theorem}[section]
\newtheorem{example}[theorem]{Example}

\newtheorem{lemma}[theorem]{Lemma}
\newtheorem{corollary}[theorem]{Corollary}

\author{
  Zvonimir Bujanovi\'c\thanks{University of Zagreb, Faculty of Science, Dept. of Mathematics, Croatia, {\tt zbujanov@math.hr}. Part of this work was done while the author was
  a postdoctoral researcher at EPF Lausanne, Switzerland.}
\and Daniel Kressner\thanks{Institute of Mathematics, EPF
  Lausanne, Switzerland, {\tt daniel.kressner@epfl.ch}.}
}

\title{Norm and trace estimation with random rank-one vectors}

\begin{document}

\maketitle

\begin{abstract}
A few matrix-vector multiplications with random vectors are often sufficient to obtain reasonably good estimates for the norm of a general matrix or the trace of a symmetric positive semi-definite matrix. Several such probabilistic estimators have been proposed and analyzed for standard Gaussian and Rademacher random vectors. In this work, we consider the use of rank-one random vectors, that is, Kronecker products of (smaller) Gaussian or Rademacher vectors. It is not only cheaper to sample such vectors but it can sometimes also be much cheaper to multiply a matrix with a rank-one vector instead of a general vector.
In this work, theoretical and numerical evidence is given that the use of rank-one instead of unstructured random vectors still leads to good estimates. In particular, it is shown that our rank-one estimators multiplied with a modest constant constitute, with high probability, upper bounds of the quantity of interest. Partial results are provided for the case of lower bounds. The application of our techniques to condition number estimation for matrix functions is illustrated.
\end{abstract}


\section{Introduction}

This work is concerned with estimating the norm of a matrix $A \in \R^{m\times n}$ or the trace of a symmetric positive semi-definite matrix $B \in \R^{n\times n}$, which are given implicitly via matrix-vector products. Given an integer factorization $n = \hat n \cdot \tilde n$, we say that $x \in \R^n$ has rank one if it takes the form $x = \tilde x\otimes \hat x$, where $\tilde x \in \R^{\tilde n}$, $\hat x \in \R^{\hat n}$ are nonzero and $\otimes$ denotes the Kronecker product. Equivalently, the $\hat n\times \tilde n$ matrix obtained from reshaping $x$ has rank one. 
In this work, we specifically target a setting where it is (much) cheaper to multiply with a rank-one vector than with a general vector. 
Short sums of Kronecker products have this property and such matrices arise in a variety of applications; see, e.g.,~\cite{Grasedyck2013,Langville2004}. More intricately, condition number estimation for matrix equations and matrix functions involves linear operators that can often be cheaply applied to rank-one vectors; see Section~\ref{sec:applications} for details. Another situation involving such operators arises in the context of error estimates for low-rank tensor approximation~\cite{KressnerPerisa}.

\begin{paragraph}{Existing work.}
 Norm estimation is a classical topic in matrix analysis; see Chapter 15 in the book by Higham~\cite{Higham2002} for a comprehensive overview.
A single matrix-vector product $Ax$ with a random vector $x$ chosen from a suitable distribution can already give a good first estimate of the norm of $A$.
A classical result by Dixon~\cite{Dixon83} shows that the spectral norm $\|A\|_2$ is bounded by $\theta \|Ax\|_2$ with probability at least $1-0.8\,\sqrt{n}\, \theta^{-1/2}$ if $x$ is distributed uniformly on the unit sphere in $\R^n$. The normalization of $x$ implies that $\|Ax\|_2$ is always a lower bound but it also leads to the unfavorable appearance of $\sqrt{n}$ in the tail probability. The latter can be avoided when choosing $x$ from $\mathcal N(0,I_n)$. This choice of $x$ is called standard Gaussian (or normal) random vector. In this case, Lemma 4.1 in~\cite{Tropp11} states that
\begin{equation} \label{eq:tropp}
 \|A\|_2 \le \theta \|Ax\|_2 \quad \text{holds with probability at least $1-\sqrt{2\pi^{-1}}\, \theta^{-1}$.}
\end{equation}
\change{Note that} the expected value of $\|Ax\|_2^2$ is not $\|A\|^2_2$ but $\|A\|^2_F$, where $\|\cdot\|_F$ denotes the Frobenius norm of a matrix. In turn, one expects that $\|Ax\|_2$ tends to overestimate $\|A\|_2$ by a factor $\|A\|_F/\|A\|_2$, which can be between $1$ and $\sqrt{n}$ depending on the singular value distribution of $A$. \change{For the setting in~\cite{Tropp11}, which considers matrices $A$ that can be well approximated by low-rank matrices, this factor is usually modest.}

There are numerous approaches to go beyond the simple estimate $\|Ax\|_2$ and construct estimators that improve upon~\eqref{eq:tropp}. A straightforward modification suggested, e.g., in~\cite{Tropp11} is to choose $k\ge 2$ independent Gaussian vectors $x_1,\ldots,x_k$ and return the maximum estimator $\max \{ \|Ax_1\|_2,\ldots, \|A x_k\|_2\}$. On the one hand, this increases the success probability in~\eqref{eq:tropp} to $1-\big( \sqrt{2\pi^{-1}}\, \theta^{-1} \big)^k$, but, on the other hand, this also increases the risk of overestimation.

A second approach starts with the observation that $\|A\|^2_F = \text{trace}(B)$, where $B = A^T A$, which suggests the use of the stochastic trace estimator
\[
 \mathsf{Est}_k = \frac1k \sum_{i = 1}^k x_i^T B x_i,
\]
going back to Hutchinson~\cite{Hutchinson1990}.
Because $\mathsf{Est}_k$ is a sub-exponential random variable with mean $\text{trace}(B)$, one can apply Chernoff bounds to obtain concentration inequalities that yield relative upper and lower bounds~\cite{Avron2011,Gratton2018,Roosta2015}. For example, the proof of Corollary 3.3 in~\cite{Gratton2018} establishes 
\begin{align} \label{eq:uppertrace}
  \text{trace}(B) & \le \frac{\mathsf{Est}_k}{1-\varepsilon} \quad \text{holds with probability at least $1-\exp(-k\rho\varepsilon^2/4)$ for all $0<\varepsilon < 1$,} \\
 \label{eq:lowertrace}
  \text{trace}(B) & \ge \frac{\mathsf{Est}_k}{1+\varepsilon} \quad \text{holds with probability at least $1-\exp(-k\rho\varepsilon^2/2)$ for all $0< \varepsilon$,}
\end{align}
with the stable rank $\rho = \|B\|_F^2 / \|B\|_2^2$. One disadvantage of~\eqref{eq:uppertrace} compared to~\eqref{eq:tropp} is that the bound on the failure probability for fixed $k$ does not converge to $0$ when loosening the upper bound (that is, when letting $\varepsilon \to 1$). For example, for $\rho = 1$, the success probability cannot be larger than $1-\exp(-k/4)\approx 1-0.78^k$. This issue is addressed in~\cite{Gratton2018} by a combination with the techniques from~\cite{Gudmundsson1995}.

The entries of a Rademacher vector are independent random variables that are either $1$ or $-1$, both with
probability $1/2$. \change{As shown in~\cite[Theorem 1]{Roosta2015}, each of the bounds~\eqref{eq:uppertrace} and~\eqref{eq:lowertrace} holds for independent Rademacher vectors $x_1,\ldots,x_k$ with probability at least $1-\exp(-k( \varepsilon^2/4 - \varepsilon^3/6))$. A more recent result~\cite[Corollary 13]{Cortinovis2020} establishes that both bounds hold jointly with probability at least $1-2\exp(-k\rho\varepsilon^2 (1+\varepsilon)^{-1}/8)$.}

A third approach to improve~\eqref{eq:tropp} is to apply a few steps of the power or Lanczos method to $B$ with a standard Gaussian starting vector; see, e.g.,~\cite{Dixon83,Hochstenbach2013,Kuczynski1992}. We will not consider such an approach in this work because repeated application of $B$ to the same vector makes it difficult to benefit from rank-one structure of the starting vector.
\end{paragraph}

\begin{paragraph}{New results.}
In this work, we develop and analyze estimators that operate with vectors $\tilde x\otimes \hat x$, where $\tilde x$ and $\hat x$ are independent random vectors. A simple calculation shows that $\|A(\tilde x\otimes \hat x)\|_2^2$ retains the property of being an unbiased estimator of $\|A\|_F^2$ for common choices of distributions for the entries of $\tilde x,\hat x$; see Lemma~\ref{lemma:expectedvalue}. For the specific case that both $\tilde x,\hat x$ are standard Gaussian vectors, Theorem~\ref{tm:norm2-left-tail}, one of our main results, shows that
\begin{equation} \label{eq:troppkron}
 \|A\|_2 \le \theta \|A(\tilde x\otimes \hat x)\|_2 \quad \text{holds with probability at least $1-2 \pi^{-1} ( 2 + \ln(1 + 2\theta) )\theta^{-1}$.}
\end{equation}
Compared to~\eqref{eq:tropp}, the success probability only becomes slightly worse. Using the maximum estimator mentioned above, this probability can be easily improved, while the risk of overestimation can be quantified with the lower bounds from Theorems~\ref{tm:norm2-right-tail} and~\ref{tm:normF-right-tail}. We also explain why it is not possible to have a result of the form~\eqref{eq:troppkron} when $\tilde x,\hat x$ are Rademacher vectors. In contrast, an upper bound of the form~\eqref{eq:uppertrace} holds for rank-one standard Gaussian \emph{and} Rademacher vectors, see Theorem~\ref{tm:traceest}. \change{We also establish lower bounds of the form~\eqref{eq:lowertrace}, see Theorems~\ref{tm:tracelowerrademacher} and~\ref{tm:tracelower}, but they come with additional unfavorable factors in the exponent.}

We would like to stress that the rank-one structure significantly complicates the theory because the techniques used in existing work on norm and trace estimation do not carry over. For example, when $x,\tilde x,\hat x$ are standard Gaussian vectors the distribution of $x$ is invariant under orthogonal transformations but the distribution of $\tilde x\otimes \hat x$ is not. The random variable $\|Ax\|_2^2$ is sub-exponential but this property, which is frequently assumed in concentration inequalities, is not enjoyed by $\|A(\tilde x\otimes \hat x)\|^2_2$.

\change{Random} rank-one vectors/measurements have been used and analyzed in compressed sensing, particularly in the context of matrix recovery~\cite{Cai2015,Chen2015,Kliesch2016,Kueng2017} and phase retrieval~\cite{Tropp2015}, \change{as well as in dimensionality reduction~\cite{Jin2019,Sun2018}. While much of the existing analyses cannot directly be used to address the questions discussed in this paper, a notable exception is recent work by Versyhnin~\cite{Ver19}; see Section~\ref{sec:largesample}. }
\end{paragraph}


\section{Small-sample estimation} \label{sec:theory}

Hutchinson~\cite{Hutchinson1990} has shown that $\|Ax\|_2^2$ is an unbiased estimator of $\|A\|_F^2$ if $x$ contains independent
random variables with mean zero and variance one, which includes standard normal and Rademacher random variables. The following lemma extends this result to rank-one vectors.

\begin{lemma} \label{lemma:expectedvalue}
 Let $A \in \R^{m\times n}$, $n = \hat n \cdot \tilde n$, and let  $\hat x$, $\tilde x$ be random vectors that contain $\hat n+\tilde n$ independent random variables with mean zero and variance one. Then $\EE\big[ \|A(\tilde x\otimes \hat x)\|_2^2 \big] = \|A\|_F^2$.
\end{lemma}
\begin{proof}
We let $a_i \in \R^{\hat n \tilde n}$ denote the $i$th column of $A^T$ and reshape $a_i$ into the $\hat n \times \tilde n$ matrix $A_i$, that is, $\text{vec}(A_i) = a_i$. By standard properties of the Kronecker product~\cite{Golub2013},
\begin{equation} \label{eq:decompblabla}
 \|A(\tilde x\otimes \hat x)\|_2^2 = \sum_{i = 1}^m ( (\tilde x\otimes \hat x)^T a_i  )^2 = \sum_{i = 1}^m ( \hat x^T A_i \tilde x)^2.
\end{equation}
Because of the assumptions on $\hat x$ we have $\EE[\hat x\hat x^T] = I_{\hat n}$ and, therefore,
\begin{align*}
 \EE[( \hat x^T A_i \tilde x)^2\big] & = \EE\big[ \tilde x^T A_i^T \hat x \hat x^T A_i \tilde x  \big] 
 = \EE_{\tilde x}\big[ \EE_{\hat x} \big[ \tilde x^T A_i^T \hat x \hat x^T A_i \tilde x  \mid \tilde x \big]\big] \\
 &= \EE_{\tilde x}\big[ \tilde x^T A_i^T  A_i \tilde x \big] = \EE_{\tilde x}\big[ \|A_i \tilde x\|_2^2 \big] = \|A_i\|_F^2.
\end{align*}
Inserted into~\eqref{eq:decompblabla}, we obtain
$\EE\big[ \|A(\tilde x\otimes \hat x)\|_2^2 \big] = \sum \|A_i\|_F^2 = \sum \|a_i\|_2^2 = \|A\|_F^2$.
\end{proof}


\subsection{One-sample estimation: Upper bound for rank-one Gaussian vectors} \label{sec:uppergauss}

We proceed with our first main result, which bounds the risk of underestimating the spectral norm when using a single Kronecker product of standard Gaussian random vectors. Note that the derivation of our result relies on an anti-concentration inequality that exploits specific properties of the distribution.


\begin{theorem}
	\label{tm:norm2-left-tail}
	Let $A \in \R^{m \times n}$ and $n = \hat{n} \cdot \tilde{n}$.
	Suppose that $\hat{x} \sim \normal(0, I_{\hat{n}})$ and $\tilde{x} \sim \normal(0, I_{\tilde{n}})$, and let $\theta > 1$. Then the inequality
	\begin{equation}
		\label{tm:norm2-left-tail:claim}
		\|A\|_2 \leq \theta \|A (\tilde{x} \otimes \hat{x})\|_2
	\end{equation}
	holds with probability at least $1 - \frac{2}{\pi} \left( 2 + \ln(1 + 2\theta) \right) \theta^{-1}$.
\end{theorem}
\begin{proof}
	Let $x = \tilde{x} \otimes \hat{x}$. We will bound the probability that~\eqref{tm:norm2-left-tail:claim} fails:
	$$
		\PP{ \|A\|_2 > \theta \|Ax\|_2 } = \PP{ \|A\|^2_2 > \theta^2 \|Ax\|^2_2 } = \PP{ \|A^T A\|_2 > \theta^2 x^T A^T A x }.
	$$
	Consider the spectral decomposition
	$
		A^T A = U \Lam U^T = \lambda_1 u_1 u_1^T + \cdots + \lambda_n u_n u_n^T, \; \lambda_1 \ge \cdots \ge \lambda_n\ge 0,
	$
	and note that
	$$
		\theta^2 x^T A^T A x
			= \theta^2 \sum_{j=1}^n \lam_j ( x^Tu_j )^2
			\geq \theta^2 \lam_1 ( x^Tu_1 )^2.
	$$
	Therefore, if $\|A^T A\|_2 > \theta^2 x^T A^T A x$, then $\|A^T A\|_2 > \theta^2 \lam_1 ( x^Tu_1 )^2$, and, since $\|A^T A\|_2 = \lam_1$, we have
	\begin{equation}
		\label{tm:norm2-left-tail:eq1}
		\PP{ \|A^T A\|_2 > \theta^2 x^T A^T A x } \leq \PP{ \lam_1 > \theta^2 \lam_1 ( x^T \change{u_1} )^2 } = \PP{ ( x^Tu_1 )^2 < \frac{1}{\theta^2}}.
	\end{equation}
	To bound the last probability, we need to control the random variable $X:=x^Tu_1$. For this purpose, reshape the vector $u_1\in \R^{\hat n \tilde n}$ as a matrix: $u_1 = \text{vec}(U_1)$ with $U_1\in \R^{\hat n \times \tilde n}$,
	and consider its singular value decomposition
	$$
		U_1 = \hat V \Sigma \tilde V^T,
		\;
		\hat V = \begin{bmatrix} \hat v_1,\ldots, \hat v_r \end{bmatrix} \in \R^{\hat{n} \times r},
		\;
		\tilde V = \begin{bmatrix} \tilde v_1, \ldots, \tilde v_r \end{bmatrix} \in \R^{\tilde{n} \times r},
		\;
		\Sigma = \text{diag}(\sigma_1, \ldots, \sigma_r) \in \R^{r \times r},
	$$
	with $r = \min\{ \hat n, \tilde n\}$ and $\sum_{i=1}^r \sigma_i^2 = \|U_1\|_F^2 = \|u_1\|_2^2 = 1$. Using $x = \tilde{x} \otimes \hat{x} = \text{vec}(\hat x \tilde x^T)$, we obtain
	$$
		x^Tu_1
			= \text{vec}(\hat x \tilde x^T)^T \text{vec}(U_1)
			= \text{trace}( \tilde x \hat x^T U_1 )
			= \hat x^T U_1 \tilde x
			= \sum_{i = 1}^r \sigma_i \hat x_i \tilde x_i,
	$$
	where $\hat x_i := \hat v_i^T \hat x$ and $\tilde x_i := \tilde v_i^T \tilde x$ are independent standard normal random variables, because $\hat{V}^T \hat{x} \sim \normal(0, I_r)$ and $\tilde{V}^T \tilde{x} \sim \normal(0, I_r)$ follow from the orthogonality of $\hat{V}$ and $\tilde{V}$. For a random variable $Z$, let $\CDF_Z$ denote its cumulative distribution function, and let $\CHAR_Z(t)$ denote its characteristic function.
	The characteristic function for the product of two \change{independent} standard normal random variables is given by $1/\sqrt{1+t^2}$ and hence
    \[
		\CHAR_{\sigma_i \hat{x}_i \tilde{x}_i}(t) = \frac{1}{\sqrt{1+ \sigma_i^2 t^2}}.
	\]
	Since $X=\sum_{i = 1}^r \sigma_i \hat x_i \tilde x_i$ is a sum of independent random variables, its characteristic function is given by
	$$
		\CHAR_X(t) = \prod_{i=1}^r \frac{1}{\sqrt{1 + \sigma_i^2 t^2}}.
	$$
	We now apply L\'{e}vy's theorem, see, e.g., \cite[Corollary 2]{Shephard1991},
	to reformulate~\eqref{tm:norm2-left-tail:eq1} in terms of characteristic functions:
	\begin{align}
		\PP{( x^Tu_1 )^2 < \frac{1}{\theta^2}}
			&= \PP{\frac{-1}{\theta} <  x^Tu_1 < \frac{1}{\theta}}
			= \CDF_X\left( \frac{1}{\theta} \right) - \CDF_X\left( \frac{-1}{\theta} \right) \nonumber \\
			&= \frac{1}{\pi} \int_{-\infty}^{\infty} \frac{\sin(t / \theta)}{t} \CHAR_X(t) \dt
			= \frac{2}{\pi} \int_{0}^{\infty} \frac{\sin(t / \theta)}{t} \prod_{i=1}^r \frac{1}{\sqrt{1 + \sigma_i^2 t^2}} \dt. \label{tm:norm2-left-tail:eq_integralr}
	\end{align}
	To bound this oscillatory integral, we use $\sin(t / \theta) \leq t / \theta$ for $t \in [0, \theta]$ and $\sin(t / \theta) \leq 1$ elsewhere. Also, using $\sum_{i=1}^r \sigma_i^2 = 1$, note that
	$$
		\prod_{i=1}^r \frac{1}{\sqrt{1 + \sigma_i^2 t^2}}
			= \frac{1}{\sqrt{1 + \sum_{i=1}^r \sigma_i^2 t^2 + \text{positive terms}}}
			\leq \frac{1}{\sqrt{1 + t^2}}.
	$$
	This gives
	\begin{align*}
		 \int_{0}^{\infty} \frac{\sin(t / \theta)}{t}
			&\prod_{i=1}^r \frac{1}{\sqrt{1 + \sigma_i^2 t^2}} \dt \\
			&=  \int_{0}^{\theta} \frac{\sin(t / \theta)}{t} \prod_{i=1}^r \frac{1}{\sqrt{1 + \sigma_i^2 t^2}} \dt
			 +  \int_{\theta}^{\infty} \frac{\sin(t / \theta)}{t} \prod_{i=1}^r \frac{1}{\sqrt{1 + \sigma_i^2 t^2}} \dt \\
			& \leq  \int_{0}^{\theta} \frac{t / \theta}{t} \prod_{i=1}^r \frac{1}{\sqrt{1 + \sigma_i^2 t^2}} \dt
			 +  \int_{\theta}^{\infty} \frac{1}{t} \prod_{i=1}^r \frac{1}{\sqrt{1 + \sigma_i^2 t^2}} \dt \\
			& \leq  \int_{0}^{\theta} \frac{1}{\theta} \frac{1}{\sqrt{1 + t^2}} \dt
			 +  \int_{\theta}^{\infty} \frac{1}{t} \frac{1}{\sqrt{1 + t^2}} \dt \\
			& =  \theta^{-1} \ln(\theta + \sqrt{1+\theta^2})
			 + \ln \left( \theta^{-1} + \sqrt{1 + \theta^{-2}} \right)  \\
			& \leq  \theta^{-1} \ln(1 + 2\theta)
			 + \ln \left( 1 + 2 \theta^{-1} \right)
			\leq \theta^{-1} \ln(1 + 2\theta)
			 + 2 \theta^{-1}   \\
			& = \theta^{-1} \left( 2 + \ln(1 + 2\theta) \right),
	\end{align*}
	which, when inserted into~\eqref{tm:norm2-left-tail:eq_integralr}, implies the claim of the theorem.
\end{proof}

Comparing the bounds~\eqref{tm:norm2-left-tail:claim} and~\eqref{eq:tropp},
the use of rank-one Gaussian vectors instead of Gaussian vectors comes with a slight penalty. The following table shows the bounds for the failure probability $\mathcal{P} := \PP{ \|A\|_2 > \theta \|A x\|_2 }$ for different values of $\theta$:
	\begin{center}
		\begin{tabular}{|c|c|c|}
			\hline
			$\theta$ &
				\begin{tabular}[x]{@{}c@{}}Gaussian vectors: \\ $\mathcal{P} \leq \sqrt{\frac{2}{\pi}} \cdot \frac{1}{\theta}$ \end{tabular} &
				\begin{tabular}[x]{@{}c@{}}Rank-one Gaussian vectors: \\ $\mathcal{P} \leq \frac{2}{\pi} \cdot \frac{1}{\theta} \cdot \left( 2 + \ln(1 + 2\theta) \right)$ \end{tabular} \\ \hline
			5        &  0.159577   & 0.559957  \\
			10       &  0.079788   & 0.321144  \\
			20       &  0.039894   & 0.181869  \\
			\change{30}       &  \change{0.026596}   & \change{0.129677}  \\
			50       &  0.015957   & 0.084226  \\ \hline
		\end{tabular}
	\end{center}
\change{Note that taking larger values of $\theta$, such as $30$ or $50$, still makes sense for applications where one is interested in obtaining only a rough estimate of the matrix norm, such as the one described in Section \ref{sec:frechet}.}

Quite trivially, Theorem \ref{tm:norm2-left-tail} can also be applied to the Frobenius norm $\|A\|_F$.
Multiplying both sides of \eqref{tm:norm2-left-tail:claim} by $\sqrt{\rho} = \|A\|_F / \|A\|_2$, it follows that
\begin{equation}
	\label{eq:normF-left-tail-stable-rank}
	\PP{ \|A\|_F \leq \theta \sqrt{\rho} \cdot \|A (\tilde{x} \otimes \hat{x})\|_2 }
		\geq 1 - \frac{2}{\pi} \left( 2 + \ln(1 + 2\theta) \right) \theta^{-1}.
\end{equation}
Given that $1 \leq \rho \leq n$, this bound suggests that, as the stable rank of $A$ increases, $\|A(\tilde{x} \otimes \hat{x})\|_2$ needs to be multiplied with a larger constant in order to remain a reliable upper bound for $\|A\|_F$. However, this seems to be entirely an artifact of our derivation of the bound. The numerical experiments in Section~\ref{example:numexp} below demonstrate that \eqref{eq:normF-left-tail-stable-rank} tends to be tight when the stable rank is small but is overly pessimistic when the stable rank is large.



\subsection{One-sample estimation: Lower bounds for rank-one Gaussian vectors} \label{sec:lowergauss}

As discussed in the introduction, the result of Lemma~\ref{lemma:expectedvalue} indicates that $\|A (\tilde{x} \otimes \hat{x}) \|_2$ tends to overestimate
$\|A\|_2$. The following result guarantees that an overestimation by a factor much larger than $\sqrt{n}$ is unlikely.
\begin{theorem}
	\label{tm:norm2-right-tail}
	Let $A \in \R^{m \times n}$, and $n = \hat{n} \cdot \tilde{n}$ with $\tilde n \le \hat n$. Suppose that $\hat{x} \sim \normal(0, I_{\hat{n}})$ and $\tilde{x} \sim \normal(0, I_{\tilde{n}})$, and let $\theta > 1$. The inequality
	\begin{equation}
		\label{tm:norm2-right-tail:claim}
		\|A\|_2 \ge n^{-1/2} \theta^{-1} \|A (\tilde{x} \otimes \hat{x}) \|_2
	\end{equation}
	holds with probability at least $1 - 2 e^{-\tilde n (\theta - \ln\theta - 1)/2}$.
\end{theorem}
\begin{proof}
	Proceeding with the spectral decomposition of $A^T A$ as in the proof of Theorem \ref{tm:norm2-left-tail}, it follows that the probability of \eqref{tm:norm2-right-tail:claim} failing for $x = \tilde{x} \otimes \hat{x}$ is
	\begin{equation}
		\label{eq:norm2-right-tail:1}
		\PP{ \|Ax\|_2 > \sqrt{n} \theta\, \|A\|_2 }
			= \PP{ \sum_{j=1}^n \lam_j (x^T u_j)^2 > \lam_1 n \theta^2 } \leq \PP{ \|x\|_2^2 > n \theta^2 },
	\end{equation}
	where the last inequality uses $\sum_{j=1}^n \lam_j (x^T u_j)^2 \leq \lam_1 \sum_{j=1}^n (x^T u_j)^2 =  \lam_1 \|x\|_2^2$.
Because $\|x\|_2 = \|\tilde x\|_2\|\hat x\|_2$ and $n = \hat{n} \cdot \tilde{n}$, it follows that
$\|x\|_2^2 > n \theta^2$ is only possible if $\|\tilde x\|_2^2 > \tilde{n}\theta$ or
$\|\hat x\|_2^2 >  \hat{n} \theta$. Thus,
\[
\PP{ \|\change{x}\|_2^2 > n \theta^2}
\leq \PP{ \|\tilde x\|_2^2 > \tilde{n} \theta } + \PP{ \|\hat x\|_2^2 > \hat{n} \theta } \le 2 \cdot \PP{ \|\tilde x\|_2^2 > \tilde{n} \theta }.
\]
Now, $\|\tilde x\|_2^2$ is a chi-square random variable with $\tilde n$ and its properties (see Lemma~\ref{lem:chi_bound} in the Appendix) imply
\[
 \PP{ \|\tilde x\|_2^2 > \tilde{n} \theta } \le (\theta e^{1-\theta})^{\tilde{n}/2},
\]
which concludes the proof.
\end{proof}

As already discussed for the upper bound, one directly obtains a corresponding result for the Frobenius norm by multiplying \eqref{tm:norm2-right-tail:claim} with $\sqrt{\rho}$:
\begin{equation}
	\label{eq:normF-right-tail-stable-rank}
	\PP{ \|A\|_F \ge \sqrt{\rho} n^{-1/2} \theta^{-1} \|A (\tilde{x} \otimes \hat{x}) \|_2 }
		\geq 1 - 2 e^{-\tilde n (\theta - \ln\theta - 1)/2}.
\end{equation}
Another bound for the Frobenius norm estimates is obtained by using a different approach.

\begin{theorem}
	\label{tm:normF-right-tail}
	Let $A \in \R^{m \times n}$ and $n = \hat{n} \cdot \tilde{n}$.
	Suppose that $\hat{x} \sim \normal(0, I_{\hat{n}})$ and $\tilde{x} \sim \normal(0, I_{\tilde{n}})$, and let $\theta > 2$. The inequality
	\begin{equation*}
		\|A\|_F \ge \theta^{-1} \|A (\tilde{x} \otimes \hat{x})\|_2
	\end{equation*}
	holds with probability at least
	$
		1 - \sqrt{2 \theta }\,e^{-\theta+2}.
	$
\end{theorem}
\begin{proof}
Using, once again, the spectral decomposition of $A^T A$ as in the proof of Theorem \ref{tm:norm2-left-tail} and
denoting $x = \tilde{x} \otimes \hat{x}$, the failure probability equals
	\begin{align*}
		\PP{ \|A x \|_2 > \theta \|A\|_F }
			&= \PP{ x^T A^T A x > \theta^2 \cdot \text{trace}(A^T A) }
			= \PP{ \sum_{i=1}^n \lam_i (x^T u_i)^2 > \theta^2 \sum_{i=1}^n \lam_i } \\
			&= \PP{ \sum_{i=1}^n \mu_i (x^T u_i)^2 > \theta^2 },
	\end{align*}
	where $\mu_i := \frac{\lam_i}{\sum_{j=1}^n \lam_j} \in [0, 1]$, which satisfy $\sum_{i=1}^n \mu_i = 1$.
	Exploiting that the function $f(\xi) = e^{\sqrt{1 + \xi}}$ is convex on $[ 0, +\infty \rangle$,
Jensen's inequality gives
	\begin{align*}
		f \left( \sum_{i=1}^n \mu_i \cdot t^2 (x^T u_i)^2 \right)
			&\leq \sum_{i=1}^n \mu_i f \left(  t^2 (x^T u_i)^2 \right)
			= \sum_{i=1}^n \mu_i e^{\sqrt{ 1 + t^2 (x^T u_i)^2}} \\
			&\leq \sum_{i=1}^n \mu_i e^{1 + t |x^T u_i|}
			= e \sum_{i=1}^n \mu_i e^{t |x^T u_i|}
	\end{align*}
	for all $t > 0$. Combined with the monotonicity of $f$ as well as Markov's inequality, we obtain
	\begin{align}
		\PP{ \sum_{i=1}^n \mu_i (x^T u_i)^2 > \theta^2 }
			&= \PP{ f \Big( \sum_{i=1}^n \mu_i \cdot t^2 (x^T u_i)^2 \Big) > f( t^2 \theta^2 ) } \nonumber \\
			&\leq \PP{ e \sum_{i=1}^n \mu_i e^{t |x^T u_i|} > e^{\sqrt{ 1 + t^2 \theta^2 } } } \nonumber \\
			&\leq \EE\Big[ \sum_{i=1}^n \mu_i e^{t |x^T u_i|} \Big] e^{ 1 - \sqrt{ 1 + t^2 \theta^2 } } \nonumber \\
			&= \Big( \sum_{i=1}^n \mu_i \EE\big[e^{t |x^T u_i|} \big] \Big) e^{ 1 - \sqrt{ 1 + t^2 \theta^2 } }. \label{eq:normF-left-tail:1}
	\end{align}
Since $e^{|a|} \leq e^a + e^{-a}$, we obtain for $0 < t < 1$ that
	$$
		\EE\big[e^{t |x^T u_i|} \big] \leq \EE\big[e^{t (x^T u_i)} \big] + \EE\big[e^{t (x^T \cdot (-u_i))} \big] \le
		2 \frac{1}{\sqrt{1-t^2}},
	$$
	where we used Corollary~\ref{corollary:gausmgf} in the last step. Plugged into~\eqref{eq:normF-left-tail:1}, it follows that
	\[
	 \PP{ \sum_{i=1}^n \mu_i (x^T u_i)^2 > \theta^2 } \le  2 e \frac{ e^{ - \sqrt{ 1 + t^2 \theta^2 } } }{\sqrt{1-t^2}} \le
	  \sqrt{2 \theta }\,e^{-\theta+2},
	\]
	where the last inequality follows from setting $t =  \sqrt{1 - 2/\theta}$.
%
%
\end{proof}

\begin{example} \label{example:numexp}
	\label{ex:small-sample}
	To illustrate the theoretical results presented above, we generate matrices with different singular value decompositions, and compare their norms with randomized estimates.
	More precisely, we consider the following seven $n \times n$ matrices:
	\change{%
	\begin{itemize}
		\item rank-one matrices $A_1 = U_1 (e_1 e_1^T)$, $A_2 = U_2 (e_1 e_1^T) V_2^T$, where $e_1 = [1, 0, \dots, 0]^T$;
		\item matrices $A_3 = U_3 D$, $A_4 = U_4 D V_4^T$, where $D$ is diagonal with $D_{ii} = e^{-i/2}$;
		\item matrices $A_5 = U_5 D$, where $D$ is diagonal with $D_{ii} = i^2$; 
		\item the matrix $A_6$ is a random orthogonal matrix.
		\item the matrix $A_7$ is a random Gaussian matrix.
	\end{itemize}
	}
	Here, $U_i$, $V_i$, \change{and $A_6$} are chosen randomly from the uniform distribution on orthogonal matrices.

	For each of these matrices, we sample $100\,000$ vectors $x = \tilde{x} \otimes \hat{x}$ for $\tilde{x}, \hat{x} \sim \normal(0, I)$ in order to estimate the probability that one of the inequalities $\|A_i\| \leq \tau \|A_i x\|_2$  or $\|A_i\| \geq \tau^{-1} \|A_i x\|_2$ fails, with $\tau \in [1, 100]$. The obtained results are shown in Figure~\ref{fig:small-sample}. We have chosen $n = 16$ and $n=196$ for $\|\cdot\|\equiv \|\cdot\|_2$ and 
	$\|\cdot\|\equiv \|\cdot\|_F$, respectively.
	
	Figure~\subref*{fig:norm2-upper} displays failure probabilities for the upper bound $\|A_i\|_2 \le \tau \|A_i x\|_2$ as well as the corresponding bound by 
	Theorem~\ref{tm:norm2-left-tail}, with $\tau = \theta$. The bound happens to be quite tight for the rank-one matrices $A_1$ and $A_2$. Although the bound is significantly less tight for the other matrices, for which the spectral and Frobenius norms are different, this demonstrates that the result of Theorem~\ref{tm:norm2-right-tail} cannot be improved significantly without taking additional properties of the matrix into account.
	
	Figure~\subref*{fig:norm2-lower} displays failure probabilities for the lower bound $\|A_i\|_2 \ge \tau^{-1} \|A_i x\|_2$ as well as the corresponding bound by 
	Theorem~\ref{tm:norm2-right-tail}, with $\tau = \sqrt{n}\, \theta$. Clearly, the bound is not sharp but it correctly captures the exponential decay of the probabilities with respect to $\theta$.
	
	Figure~\subref*{fig:normF-upper} displays failure probabilities for the upper bound $\|A_i\|_F \le \tau \|A_i x\|_2$ as well as the corresponding bound from~\eqref{eq:normF-left-tail-stable-rank}, with $\tau = \sqrt{\rho}\,\theta$. As the bound depends on the stable rank, the dashed lines shown in the figure differ for matrices with different stable rank. Solid and dashed lines of the same color and mark belong to the same matrix $A_i$. Again, the bounds are quite tight for 
	low (stable) rank but becomes increasingly loose as the stable rank increases. In fact, the bound increases with larger stable rank while estimated failure probabilities actually \emph{decrease}. 
	
    Figure~\subref*{fig:normF-lower} displays failure probabilities for the upper bound $\|A_i\|_F \le \tau^{-1} \|A_i x\|_2$ as well as the corresponding bound~\eqref{eq:normF-right-tail-stable-rank}, with $\tau = \sqrt{n/\rho}\, \theta$. Again, this bound depends on the stable rank. Although far off for matrices of low (stable) rank, it correctly captures the observation that the exponential decay becomes faster as the stable rank $\rho$ increases. In contrast, the bound from Theorem~\eqref{eq:normF-right-tail-stable-rank} does not depend on $\rho$ and is clearly preferable when $\rho$ is small or no estimate of $\rho$ is available.

\end{example}

\begin{figure} 
	\centering
	\subfloat[%
		Theorem \ref{tm:norm2-left-tail} (dashed line) vs. estimated failure probabilities for
		$\|A_i\|_2 > \tau \|A_i x\|_2$ (solid line).%
	]{%
		\label{fig:norm2-upper}\includegraphics[width=0.49\textwidth]{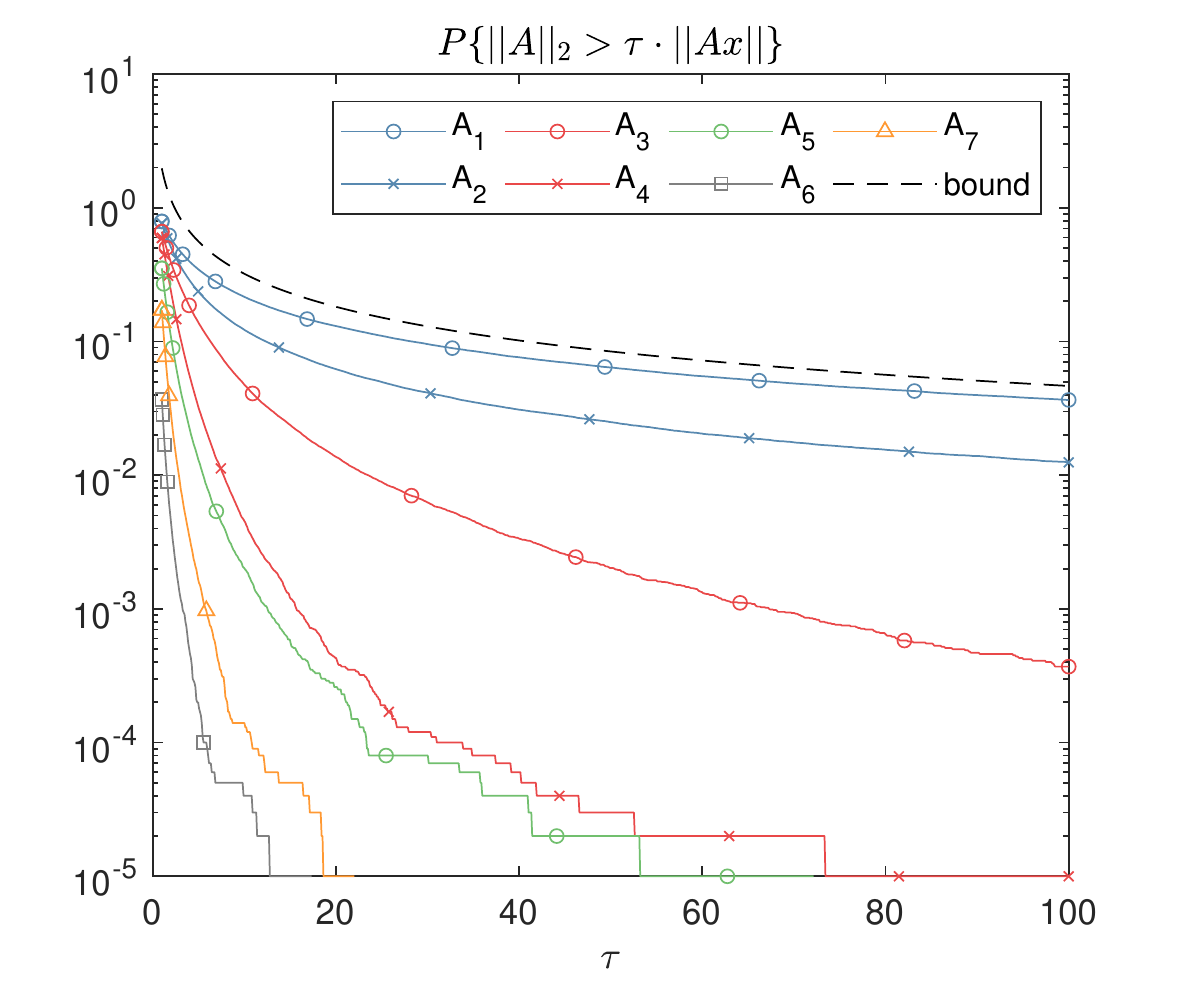}}%
	\hfill
	\subfloat[%
		Theorem \ref{tm:norm2-right-tail} (dashed line) vs. estimated failure probabilities for
		$\|A_i\|_2 < \tau^{-1} \|A_i x\|_2$ (solid line).%
	]{%
		\label{fig:norm2-lower}\includegraphics[width=0.49\textwidth]{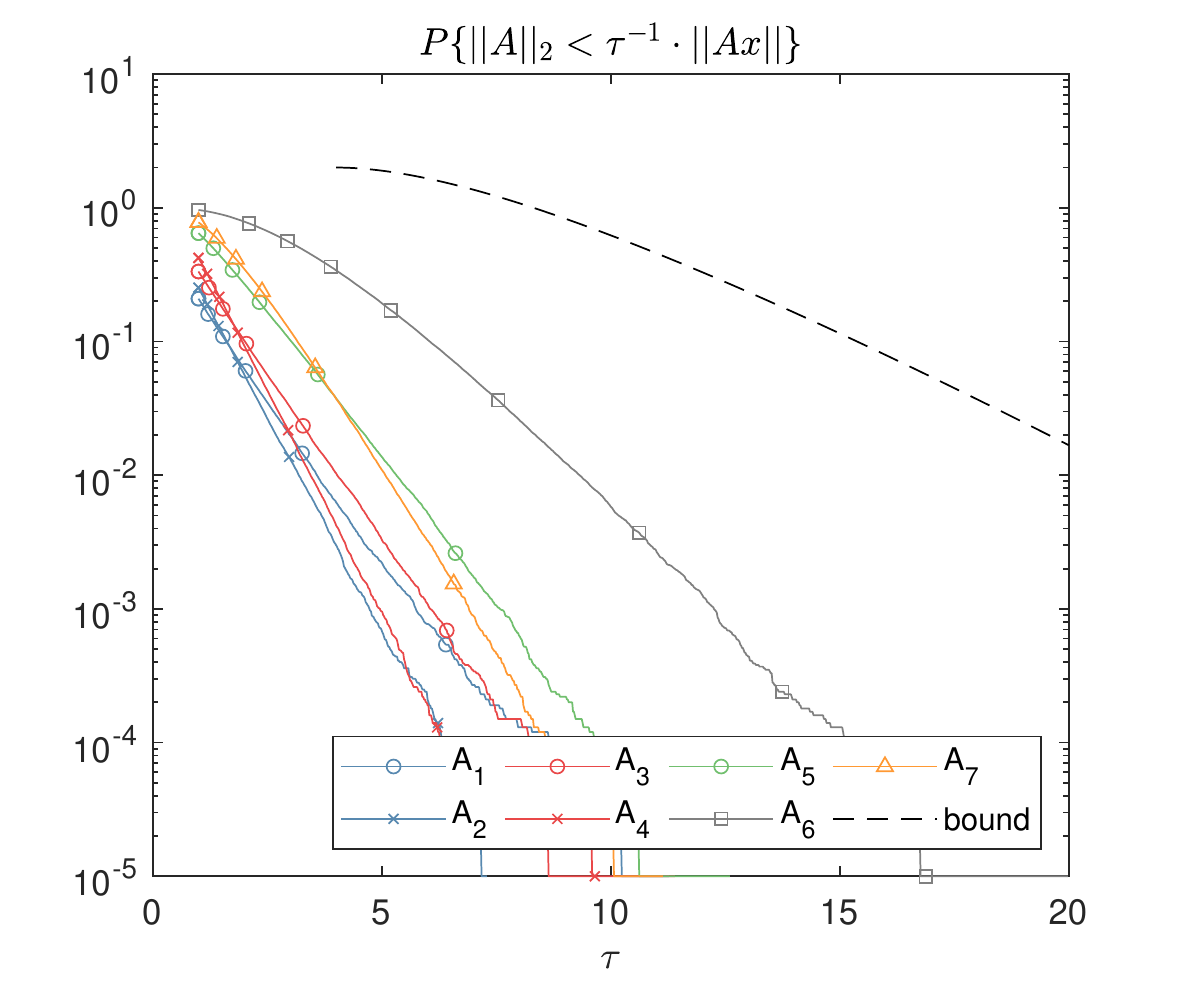}}%
	\\
	\subfloat[%
		Bound~\eqref{eq:normF-left-tail-stable-rank} (dashed lines) vs. estimated failure probabilities for
		$\|A_i\|_F > \tau \|A_i x\|_2$ (solid lines).
	]{%
		\label{fig:normF-upper}\includegraphics[width=0.49\textwidth]{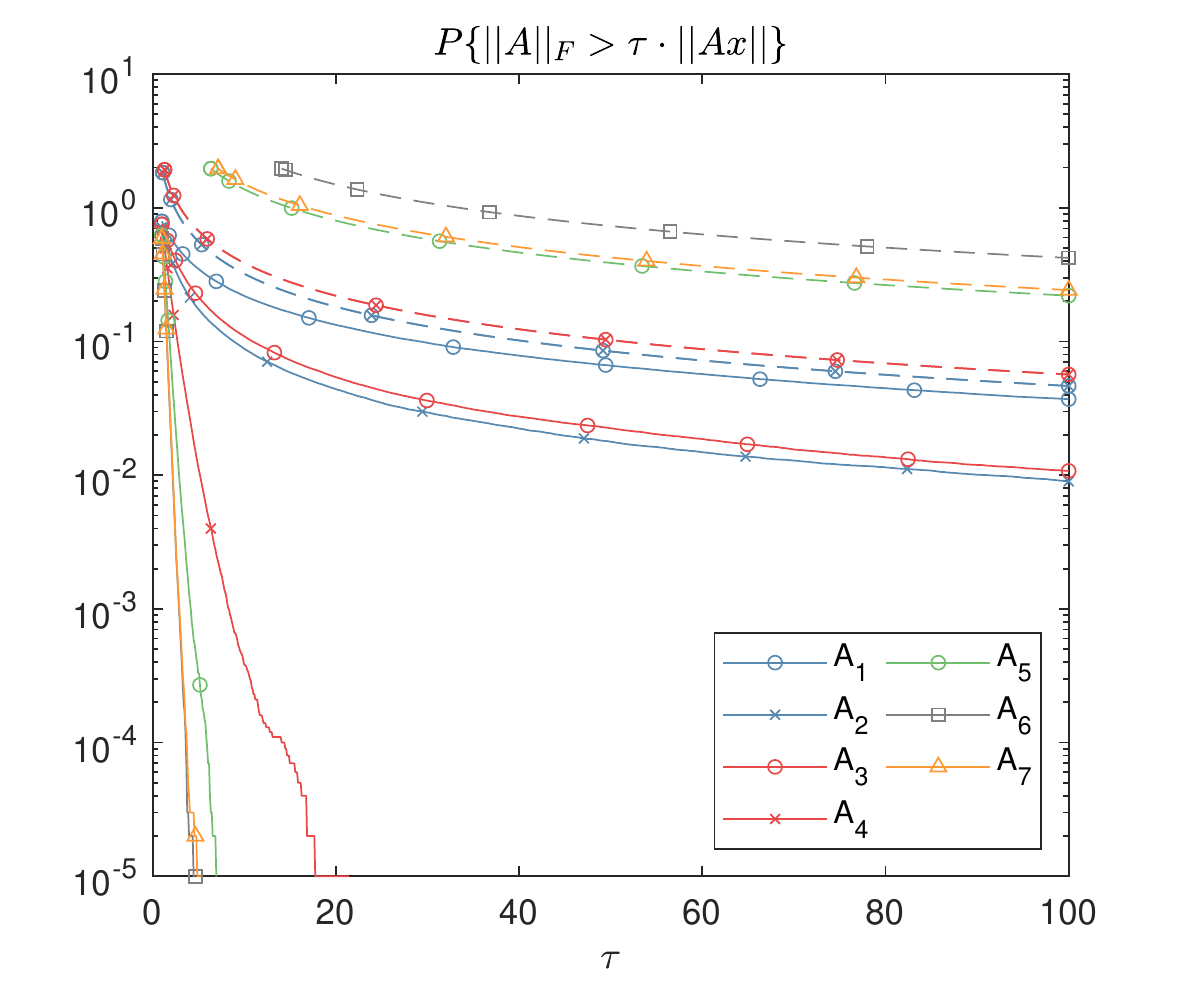}}%
	\hfill
	\subfloat[%
		Theorem~\ref{tm:normF-right-tail} (dashed black line) versus estimated failure probabilities for
		$\|A_i\|_F < \tau^{-1} \|A_i x\|_2$. Upper bounds~\eqref{eq:normF-right-tail-stable-rank} are also shown (dashed colored lines).
	]{%
		\label{fig:normF-lower}\includegraphics[width=0.49\textwidth]{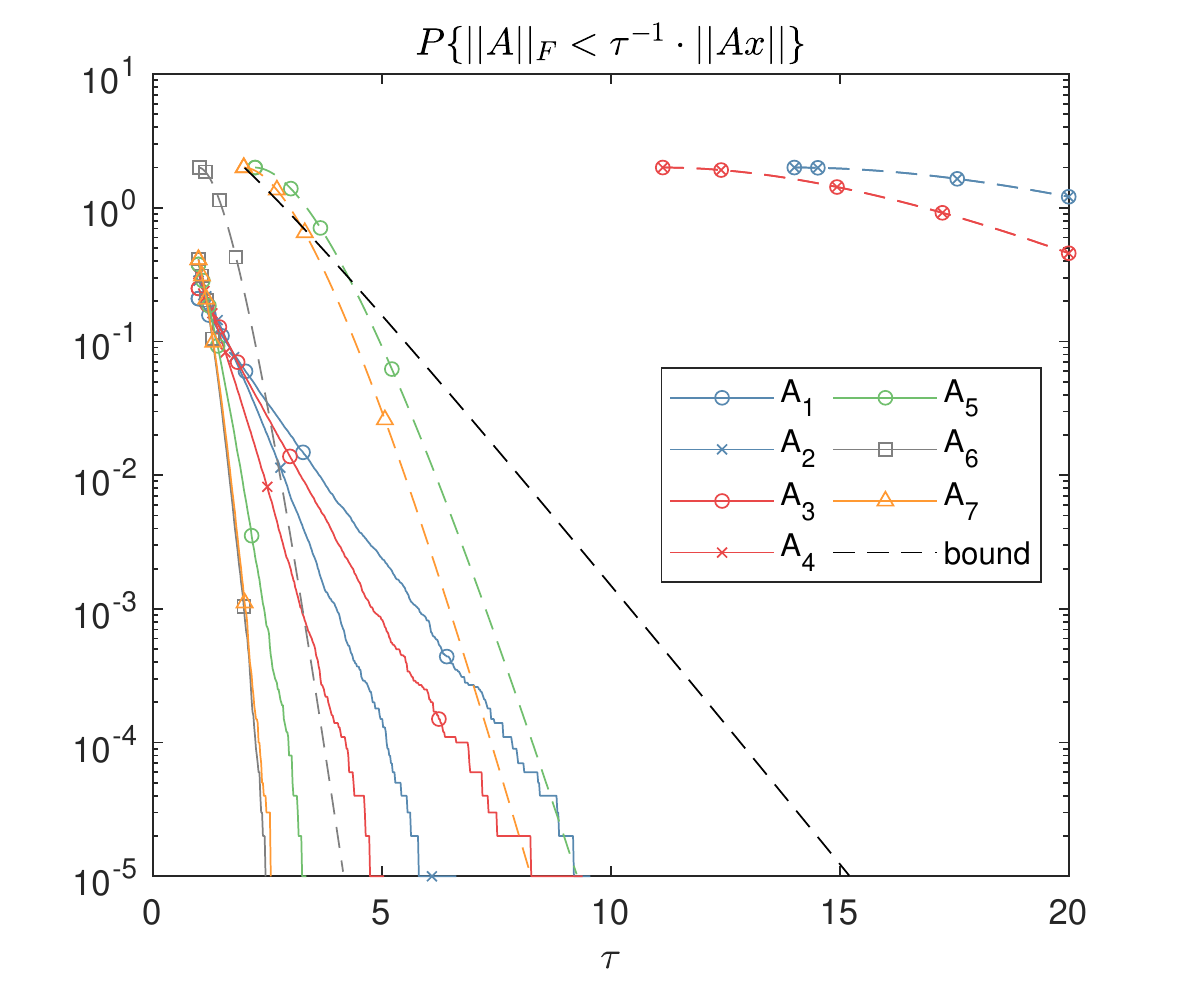}}
	\caption{%
        \label{fig:small-sample}
		Theoretical bounds on failure probabilities versus estimates; see Example \ref{ex:small-sample} for details.
		}
\end{figure}

\subsection{Small-sample estimation: Rank-one Gaussian vectors}
\label{sec:small-sample}

As discussed in the introduction, a simple way to reduce the failure probability of upper bounds is to use a maximum estimator. In our setting, this translates into
\begin{equation*} 
 \mathsf{Max}_k:=\max_{j=1, \ldots, k} \|A (\tilde{x}_j \otimes \hat{x}_j)\|_2
\end{equation*}
with independent standard Gaussian vectors $\tilde{x}_j$, $\hat{x}_j$ for $j=1, \ldots, k$.

Using Theorem~\ref{tm:norm2-left-tail} one obtains
\begin{equation}
	\label{eq:norm2-upper-max-estimator}
	\PP{ \|A\|_2 \leq \theta \cdot \mathsf{Max}_k }
		\geq 1 - \Big( \frac{2}{\pi} \left( 2 + \ln(1 + 2\theta) \right) \theta^{-1} \Big)^k.
\end{equation}
For example, for $\theta = 10$, choosing $k = 7$ is sufficient to guarantee a success probability of more than $99.9\%$.

Of course, taking the maximum for several samples increases the risk of overestimation. But this increase can be easily mitigated by accepting a slight increase of the overestimation factor.
By Theorem~\ref{tm:normF-right-tail}, the probability that $\|A (\tilde{x} \otimes \hat{x})\|_2$ overestimates $\|A\|_F$ by more than a factor $10$ is less than $0.15\%$. For the maximum estimator one has (again by Theorem~\ref{tm:normF-right-tail}) that
$$
	\PP{ \|A\|_F \ge \theta^{-1}\cdot \mathsf{Max}_k }
		\geq \big( 1 - \sqrt{2 \theta }\,e^{-\theta+2} \big)^k.
$$
For $k = 7$, the probability that $\mathsf{Max}_k$ overestimates $\|A\|_F$ by more than a factor $12.1$ is again less than $0.15\%$.

\subsection{One-sample estimation: Rank-one Rademacher vectors} \label{sec:rademacherone}

In this section, we discuss whether the results from Sections~\ref{sec:uppergauss} and~\ref{sec:lowergauss} can be extended when we choose $\tilde{x}$ and $\hat{x}$ to be Rademacher instead of standard Gaussian vectors in the rank-one vector $x = \tilde{x} \otimes \hat{x}$.


%

It turns out that it is not possible to have an upper bound of the form presented in Theorem~\ref{tm:norm2-left-tail} for Rademacher vectors \change{without further assumptions on $A$}. To see this, consider the matrix $A = uu^T \in \R^{n \times n}$ with $u = e / \sqrt{n}$, where $e$ denotes the vector of all ones.
Then $A^T A = A$ and
$$
	\PP{ \|A\|_2 > \theta \|Ax\|_2 }
		= \PP{ \|A^T A\|_2 > \theta^2 x^T A^T A x}
		= \PP{ (x^T u)^2 < \theta^{-2} }
		= \PP{ -\theta^{-1} < x^T u < \theta^{-1} }.
$$
Now $x^T u = (\tilde x\otimes \hat x)^T u = n^{-1/2} (\hat x^T e) (\tilde x^T e )$, which implies
$$
	\PP{ x^T u = 0 }
		= \PP{ (\hat x^T e) (\tilde x^T e ) = 0 }
		\geq \PP{ \hat x^T e = 0 }.
$$
For even $\hat{n}$, the sum $\hat x^T e = \hat x_1 + \cdots  + \hat x_{\hat n}$ equals zero when exactly $\hat{n}/2$ of the entries are
equal to~$1$. As this happens with probability $\binom{ \hat{n} }{ \hat{n}/2 } / 2^{\hat{n}}$, we obtain the lower bound
$$
\PP{ \|A\|_2 > \theta \|Ax\|_2 } \ge  \PP{ x^T u = 0 }
		\geq 2^{-\hat{n}} \binom{ \hat{n} }{ \hat{n}/2 }
		\approx \frac{1}{\sqrt{\hat{n} \pi/2}},
$$
for any $\theta > 1$. In particular, and in contrast to the result of Theorem~\ref{tm:norm2-left-tail}, the failure probability does not converge to zero as $\theta$ increases. Let us stress that this is not an artifact of using rank-one vectors; an analogous negative result can be obtained when using an (unstructured) Rademacher vector $x$.

On the other hand, because Rademacher vectors are bounded, the risk of overestimation becomes zero beyond a certain threshold. More specifically, it always holds that $\|\hat{x}\|_2 = \sqrt{\hat{n}}$, $\|\tilde{x}\|_2 = \sqrt{\tilde{n}}$, and hence, for all matrices $A$, by equation \eqref{eq:norm2-right-tail:1} we have
$$
	\PP{ \|Ax\|_2 > \sqrt{n} \|A\|_2 }
		\leq \PP{ \|x\|_2 > \sqrt{n}  }
		= \PP{ \|\hat{x}\|_2 \cdot \|\tilde{x}\|_2 > \sqrt{n} } = 0.
$$
This is in contrast to the result of Theorem~\ref{tm:norm2-right-tail}, which yields a small but nonzero risk for Gaussian vectors.

\section{Large-sample estimation of the trace} \label{sec:largesample}

In this section, we analyze the use of random rank-one vectors in stochastic trace estimators for estimating $\trace(B) = \|A\|_F^2$ with $B = A^T A$:
\begin{equation} \label{eq:tracestimator}
 \mathsf{Est}_k = \frac1k \sum_{i = 1}^k (\tilde x_i \otimes \hat x_i)^T B ( \tilde x_i \otimes \hat x_i ).
\end{equation}
By Lemma~\ref{lemma:expectedvalue}, we have $\EE[\mathsf{Est}_k] = \trace(B)$. The following theorem shows that $\mathsf{Est}_k$ times a modest factor is an upper bound of $\trace(B)$ with high probability for larger $k$.

\begin{theorem}\label{tm:traceest}
Let $0<\varepsilon < 1$ and consider the trace estimator $\mathsf{Est}_k$ defined in~\eqref{eq:tracestimator}, where $\tilde x_i$, $\hat x_i$ are either independent standard Gaussian or independent Rademacher random vectors. Then the bound
  \[
   \trace(B) \le \frac{1}{1-\varepsilon} \mathsf{Est}_k
  \]
  holds with probability at least $1-\exp(-k\varepsilon^2/18)$.
\end{theorem}
\begin{proof}
The proof is divided into two parts. First, we follow the arguments from~\cite{Roosta2015} and use a Chernoff bound. We then arrive at the problem of bounding the moment generating function of decoupled Gaussian / Rademacher chaos, which will be discussed in the second part.

Consider the spectral decomposition $B = U \Lambda U^T$ with $U$ orthogonal and $\Lambda$ diagonal containing the eigenvalues $\lambda_1\ge \cdots \ge \lambda_n \ge 0$ on the diagonal. Letting $u_j$ denote the $j$th column of $U$ and $x_i = \tilde x_i \otimes  \hat x_i$, we obtain
\[
 \mathsf{Est}_k = \frac1k \sum_{i = 1}^k x_i^T U \Lambda U^T x_i = \frac1k \sum_{i = 1}^k \sum_{j = 1}^n \lambda_j (x_i^T \change{u_j})^2 =
 \frac1k  \sum_{j = 1}^n \lambda_j \sum_{i = 1}^k z_{ij}^2,
\]
where we set $z_{ij} := x_i^T \change{u_j}$.
The statement of the theorem is equivalent to showing that $\exp(-k\varepsilon^2/18)$ is an upper bound for
\begin{align*}
P &:= \PP{\mathsf{Est}_k < (1-\varepsilon) \trace(B)} =  \mathbb P\Big\{ \sum_{j = 1}^n \frac{\lambda_j}{\trace(B)} \sum_{i = 1}^k z_{ij}^2 < k (1-\varepsilon)\Big\}.
\end{align*}
By the Chernoff bound, it holds for arbitrary $t > 0$ that
\begin{align*}
 P & \le \exp(tk(1-\varepsilon)) \EE\Big[ \exp\Big( -\sum_{j = 1}^n \frac{\lambda_j}{\trace(B)} \sum_{i = 1}^k tz_{ij}^2 \Big) \Big] \\
 &\le \exp(tk(1-\varepsilon)) \sum_{j = 1}^n \frac{\lambda_j}{\trace(B)} \EE\Big[ \exp\Big( - \sum_{i = 1}^k tz_{ij}^2 \Big) \Big] \\
 &=  \exp(tk(1-\varepsilon)) \sum_{j = 1}^n \frac{\lambda_j}{\trace(B)} \prod_{i = 1}^k \EE\big[ \exp\big(-tz_{ij}^2\big)\big],
\end{align*}
where we used Jensen's inequality and the convexity of the exponential in the second inequality, and the independence of $z_{ij}$ for different $i$ in the equality.

It remains to bound the moment generating function $\EE\big[ \exp\big(-tz_{ij}^2\big)\big]$. To simplify the notation we let $Q\in \R^{\hat n \times \tilde n}$ denote the matrix such that $\text{vec}(Q) = u_j$ and drop indices. Then
$
 z = \hat x^T Q \tilde x
$, a random variable that is sometimes called decoupled (order-$2$) chaos.
Using that $\exp(-t\alpha^2) \le 1-t\alpha^2 + \frac12 t^2 \alpha^4$ holds for any fixed $\alpha\in \R$, we obtain
\[
 \EE\big[ \exp(-tz^2)\big] \leq 1- t \EE[z^2] + \frac{t^2}{2} \EE[z^4].
\]
For both, Rademacher and Gaussian random vectors $\hat{x}$ and $\tilde{x}$, we have that
\[
 \EE[z^2] = \|Q\|_F^2, \quad \EE[z^4]  \le 9  \|Q\|_F^4.
\]
For the Rademacher case, this follows from Khintchine inequalities; see, e.g.,~\cite[Section 6.8]{Rauhut2010}.
For the Gaussian case, see Lemma~\ref{lemma:gaussianchaos} in the appendix.
Because of $\|\change{u_j}\|_2 = 1$, the matrix $\change{Q}$ has Frobenius norm $1$ and we thus arrive at
\begin{align*}
 P &\le \exp(tk(1-\varepsilon)) \Big(1- t + \frac{9}{2} t^2\Big)^k.
\end{align*}
Taking the logarithm and applying Taylor expansion, we obtain
\[
 \frac1k \log P \le t(1-\varepsilon) + \log \Big(1- t + \frac{9}{2} t^2\Big) \le t(1-\varepsilon) + \Big( - t + \frac{9}{2} t^2 \Big).
\]
For $t = \varepsilon/9$, the right-hand side equals $-\varepsilon^2/18$, which concludes the proof.
\end{proof}

\change{When $\tilde n = 1$ or $\hat n = 1$ and $\tilde x_i$, $\hat x_i$ are Rademacher vectors, the Kronecker product $\tilde x_i \otimes \hat x_i$ is again a Rademacher vector and Theorem~\ref{tm:traceest} becomes a result on the standard stochastic trace estimator. Compared to the bound $\exp(-k( \varepsilon^2/4 - \varepsilon^3/6))$ on the failure probability established in~\cite{Roosta2015} for this case, the bound of Theorem~\ref{tm:traceest} is only modestly worse. Compared to the bound from~\cite{Cortinovis2020}, it can become significantly worse when the stable rank of $A$ is known a priori to be large.}

Deriving a lower bound estimate for $\mathsf{Est}_k$ turns out to be more difficult. The following theorem only consider the case of Rademacher vectors; it is shown that $n^{-1}\cdot \mathsf{Est}_k$ times a modest factor is a lower bound of $\trace(B)$ with high probability.

\begin{theorem}\label{tm:tracelowerrademacher}
Let $\varepsilon > 0$ and consider the trace estimator $\mathsf{Est}_k$ defined in~\eqref{eq:tracestimator} for  independent Rademacher random vectors $\tilde x_i$, $\hat x_i$. Then the bound
\[
   \trace(B) \ge \frac{1}{\change{1+\varepsilon}} \mathsf{Est}_k
\]
holds with probability at least $1-\exp(-k\change{\varepsilon/(n-1)})$, provided that $\change{n-1} \ge 48\cdot \varepsilon^{-1}$.
\end{theorem}
\begin{proof} Along the lines of the first part of the proof of Theorem~\ref{tm:traceest}, it can be shown that the statement of the theorem is equivalent to showing that $\exp(-k\varepsilon)$ is an upper bound on
\[
 \mathbb P\Big\{ \sum_{j = 1}^n \frac{\lambda_j}{\trace(B)} \sum_{i = 1}^k z_{ij}^2 > k(\varepsilon + 1)\Big\} =
 \mathbb P\Big\{ \sum_{i = 1}^k X_i > k  \varepsilon\Big\}.
\]
Here,
\[
 X_i = \sum_{j = 1}^n \frac{\lambda_j}{\trace(B)}  z_{ij}^2 - 1.
\]
are zero-mean independent random variables that are bounded; using $|z_{ij}| \le |\hat x_i|^T |U_j| |\tilde x_i| \le \sqrt{n}$ with $\text{vec}(U_j) = u_j$ one obtains
$|X_i| \le n-1$. Using that $\EE[z_{ij}^2] = 1$ and $\EE[z_{ij}^4] = 9$, it can be shown that $\EE[X_i^2] \le 8$.
Plugging these bounds into Bernstein's inequality completes the proof:
\[
  \mathbb P\Big\{ \sum_{i = 1}^k X_i > k  \varepsilon\Big\} \le \exp\Big[- \frac{  k^2\varepsilon^2 / 2 }{ 8k +  (n-1) k\varepsilon / 3 } \Big] \le \exp( - k\change{\varepsilon/(n-1)} ) ,
\]
where we used the imposed condition on \change{$n$} in the second inequality.
\end{proof}

\newcommand{\Est}{\mathsf{Est}_k}

\change{%
The proof technique of Theorem~\ref{tm:tracelowerrademacher} does not extend to the Gaussian case.
For $k = 1$, Vershynin~\cite{Ver19} has derived two-sided bounds for the more general setting that the vectors $\tilde x_1,\hat x_1$ have i.i.d.\@ entries from a sub-Gaussian distribution. In the following, we will use the techniques from~\cite{Ver19} to establish lower bounds that, one the one hand, apply to general $k\ge 1$ and, on the other hand, provide specific constants for Gaussian and Rademacher vectors. For this purpose, we will make use of the following result on coupled second-order Gaussian and Rademacher chaos, which can be extracted from the proofs of~Lemma 1 in \cite{Laurent2000} 
and Theorem 8 in~\cite{Cortinovis2020}, respectively.

\begin{lemma} \label{lemma:chaos}
Let $B \in \R^{n\times n}$ be symmetric positive semi-definite. Then 
\[
 \log \EE\big[\exp\!\big( t (x^T B x - \trace B) \big) \big] \le \frac{c t^2 \|B\|_F^2 }{1-2c t\|B\|_2}, \quad \text{for} \quad 0\le t < \frac{1}{2c\|B\|_2},
\]
holds with $c = 1$ when $x$ is a standard Gaussian random vector and with $c = 2$ when $x$ is a Rademacher vector.
\end{lemma}

In the following, we will use the weaker bound $2 c t^2 \|B\|_2 \cdot \trace B$, which holds for  $t \le 1 / (4c\|B\|_2)$.

\begin{theorem}
    \label{tm:tracelower}
    Consider the trace estimator $\mathsf{Est}_k$ defined in~\eqref{eq:tracestimator} and suppose that $\tilde n \le \hat n$.
    Then the bound
    \[
       \trace(B) \ge \frac{1}{1 + \varepsilon} \mathsf{Est}_k
    \]
    holds for independent standard Gaussian random vectors $\tilde x_i$, $\hat x_i$ with probability at least
    \begin{equation} \label{eq:gaussianbound}
     1 - \exp\Big(- \frac{k\rho\varepsilon^2}{50 \hat{n} } \Big) - k \frac{1}{5^{\hat{n}/2}}, \quad \text{if} \quad 0 < \varepsilon < \frac{25}{16},
    \end{equation}
    and for independent Rademacher vectors $\tilde x_i$, $\hat x_i$ with probability at least
        \begin{equation} \label{eq:rademacherbound}
     1 - \exp\Big(- \frac{k\rho\varepsilon^2}{52 \hat{n} } \Big), \quad \text{if} \quad 0 < \varepsilon < \frac{13}{4}.
        \end{equation}
Here, $\rho = \trace{B} / \|B\|_2$ is the stable rank of $A$.
\end{theorem}
\begin{proof}
We first consider the Gaussian case. As the proof follows closely the arguments in~\cite{Ver19}, we will keep it relatively brief. The main idea of~\cite{Ver19} is to use marginalization to separate $\tilde x_i$ from $\hat x_i$. In order to express this conveniently, we let $\tilde{\bf x}=(\tilde{x}_1, \ldots, \tilde{x}_k)$, $\hat{\bf x}=(\hat{x}_1, \ldots, \hat{x}_k)$, and set $f(\tilde{\bf x}, \hat{\bf x}) :=  \Est$.

    Considering the event $\mathcal{E}=\{(\tilde{\bf x}, \hat{\bf x}) \;:\; \|\hat{x}_j\| \leq 2 \hat{n}^{1/2}, \; \text{ for all } j=1, \ldots, k \}$, we have that
    \begin{align}
        \PP{ \Est - \trace{B} \geq \theta }
            &\leq \PP{ \Est - \trace{B} \geq \theta \text{ and } \mathcal{E} } + \PP{\mathcal{E}^c} \nonumber \\
            &= \PP{e^{t (\Est - \trace{B}) }\geq e^{t \theta} \text{ and } \mathcal{E} } + \PP{\mathcal{E}^c} \nonumber \\
            &= \PP{e^{t (\Est - \trace{B}) } \cdot \mathbf{1}_\mathcal{E} (\tilde{\bf x}, \hat{\bf x}) \geq e^{t\theta} } + \PP{\mathcal{E}^c} \nonumber \\
            &\leq e^{-t \theta} \EE[e^{t (\Est - \trace{B}) } \cdot \mathbf{1}_\mathcal{E}(\tilde{\bf x}, \hat{\bf x}) ] + \PP{\mathcal{E}^c},
            \label{eq:large:gaussnew}
    \end{align}
    where $\mathbf{1}$ denotes the indicator function. By Lemma~\ref{lem:chi_bound}, $
        \PP{\mathcal{E}^c} \leq k(0.2)^{\hat{n}/2}
    $.

    The exponent in~\eqref{eq:large:gaussnew} is separated via conditional expectation:
    \begin{align*}
     \Est - \trace{B} &= f(\tilde{\bf x}, \hat{\bf x}) - \EE [f(\tilde{\bf x}, \hat{\bf x})] 
     = f(\tilde{\bf x}, \hat{\bf x}) - \EE_{\tilde{\bf x}} [f(\tilde{\bf x}, \hat{\bf x})] + \EE_{\tilde{\bf x}} [f(\tilde{\bf x}, \hat{\bf x})] - \EE [f(\tilde{\bf x}, \hat{\bf x})] \\
     &= f(\tilde{\bf x}, \hat{\bf x}) - \EE_{\tilde{\bf x}} [f(\tilde{\bf x}, \hat{\bf x})] + \hat f (\hat{\bf x}) - \EE_{\hat{\bf x}} [\hat f (\hat{\bf x})], \quad \hat f (\hat{\bf x}):= \EE_{\tilde{\bf x}} [f(\tilde{\bf x}, \hat{\bf x})].
    \end{align*}
    We now consider $\hat{\bf x}$ fixed such that $\|\hat{x}_j\|_2 \leq 2\hat{n}^{1/2}$ for $j=1, \dots, k$. Viewing $\tilde{\bf x}$ as a Gaussian vector of length $k\tilde n$  we can write
    \[
    f(\tilde{\bf x}, \hat{\bf x}) = \tilde{\bf x}^T \tilde B \tilde{\bf x}, \quad
    \tilde B := \frac 1k \text{diag}\big( (I \otimes \hat x_1)^T B ( I \otimes \hat x_1 ), \ldots, (I \otimes \hat x_k)^T B ( I \otimes \hat x_k ) \big).
    \]
    Noting that $\trace \tilde B = \EE_{\tilde{\bf x}} [f(\tilde{\bf x}, \hat{\bf x})]$ and $\|\tilde B\|_2 \le \frac4k \hat{n} \|B\|_2$, Lemma~\ref{lemma:chaos} applied to $\tilde{\bf x}^T \tilde B \tilde{\bf x}$  implies for all $0 \le t \leq \frac{k}{16 \hat{n} \|B\|_2}$ that
    \[
     \log \EE_{\tilde{\bf x}} \big[\exp\!\big( t (f(\tilde{\bf x}, \hat{\bf x}) - \EE_{\tilde{\bf x}} [f(\tilde{\bf x}, \hat{\bf x}) ] \big) \big] \le 
     \frac{8 t^2 \hat n }{k}  \|B\|_2 \cdot \EE_{\tilde{\bf x}} [f(\tilde{\bf x}, \hat{\bf x})]
    \]
    or, equivalently,
    \[
     \log \EE_{\tilde{\bf x}}\big[\exp\!\big( t f(\tilde{\bf x}, \hat{\bf x}) - t_1 \EE_{\tilde{\bf x}} [f(\tilde{\bf x}, \hat{\bf x}) ] \big) \big] \le 
0
    \]
    with $t_1 = t + \frac{8t^2 \hat{n}}{k} \|B\|_2$.
    Because $\hat f(\hat{\bf x}) = \frac1k \sum_{i = 1}^k \hat x_i^T \hat B \hat x_i$, where $\hat B$ is obtained from $B$ by partitioning it into $\hat n\times \hat n$ blocks and summing up the diagonal blocks, we obtain in an analogous fashion
    \[
     \log \EE_{\hat{\bf x}}\big[\exp\!\big( t_1 (\hat f(\hat{\bf x}) - \EE_{\hat{\bf x}} [\hat f(\hat{\bf x} ) ] \big) \big] \le 
     \frac{2 t_1^2 \tilde n }{k}  \|B\|_2 \cdot \trace B.
    \]
    This requires $0 \le t_1 \leq \frac{k}{4 \tilde{n} \|B\|_2}$, which is implied by $\tilde n\le \hat n$ and the condition on $t$. Lemma 4.1 in~\cite{Ver19} now allows us to merge both bounds and obtain 
    \[
    \log \EE\big[e^{t \Est - t_1 \trace{B}) } \cdot \mathbf{1}_\mathcal{E}(\tilde{\bf x}, \hat{\bf x}) \big] \le \frac{2 t_1^2 \tilde n }{k}  \|B\|_2 \cdot \trace B.
    \]
    Using $t_1\le \frac32 t$, this implies 
    $
    \log \EE\big[e^{t ( \Est - \trace{B}) } \cdot \mathbf{1}_\mathcal{E}(\tilde{\bf x}, \hat{\bf x}) \big] \le \frac{25 t^2 \hat n }{2k}  \|B\|_2 \cdot \trace B,
    $
which, when plugged into \eqref{eq:large:gaussnew}, gives
    $$
        \PP{ \Est - \trace{B} \geq \theta }
            \leq e^{-t \theta + t^2 \cdot \frac{25}{2k} \hat{n} \|B\|_2 \cdot \trace B} + k(0.2)^{\hat{n}/2}.
    $$
    Choosing the minimizing $t = \frac{\theta k}{25 \hat{n} \|B\|_2 \cdot \trace B}$, which is valid for $\theta \leq 25/16 \cdot \trace B$, and setting
    $\varepsilon = \theta / \trace{B}$ yields the desired result~\eqref{eq:gaussianbound}.
    
    The result~\eqref{eq:rademacherbound} for Rademacher vectors is proven in the same fashion, with the simplification that $\|\hat{x}_j\|_2 = \hat n^{-1/2}$ holds with probability $1$.
    \end{proof}
    
    A few remarks on the results of Theorem~\ref{tm:tracelower} are in order. Although the second term in~\eqref{eq:gaussianbound} increases with $k$, it can be expected to  remain negligible for all practical purposes; norm and trace estimates are primarily of interest for larger values of $\hat n$. In the Rademacher case, the bound~\eqref{eq:rademacherbound} of Theorem~\ref{tm:tracelower} features the factor $\rho/(50 \hat{n})$ in the exponent, which appears to be more favorable than the corresponding factor $1/n$ in Theorem~\ref{tm:tracelowerrademacher}. On the other hand, even for small values of $k$ a rough lower bound with high probability can be obtained via Theorem~\ref{tm:tracelowerrademacher} by choosing $\varepsilon$ sufficiently large. This is not possible with the results of Theorem~\ref{tm:tracelower} because of the imposed restrictions on $\varepsilon$.}


\section{Numerical Experiments} \label{sec:applications}

\subsection{Performance of standard and rank-one estimators} \label{sec:testmatrices}

The purpose of this section is to numerically assess the impact of using rank-one random vectors instead of unstructured random vectors on the performance of the trace estimator $\mathsf{Est}_k$; see~\eqref{eq:tracestimator}. For this purpose, the following $8$ examples have been chosen to illustrate different aspects of the estimators.
\begin{description}
 \item[{\sf Ones},] matrix of all ones; $n = 2\,500$, $\tilde n = \hat n = 50$, estimation of $\trace(A) = 2\,500$.
 \item[{\sf Rank-one,}] matrix $v v^T$ where $v = \mathrm{vec}(I_{\tilde n})$ with the identity matrix $I_{\tilde n}$;
 $n = 2\,500$, $\tilde n = \hat n = 50$, estimation of $\trace(A) = 50$.
 \item[{\sf ACTIVSg2000,}] from the SuiteSparse Matrix Collection~\cite{Davis2011} originating from a synthetic electric grid model~\cite{Birchfield2017};
 $n = 4\,000$, $\tilde n = 80$, $\hat n = 50$, estimation of $\|A^{-1}\|_F^2 = \trace( A^{-T} A^{-1}  ) \approx 1.5\times 10^{4}$.
 \item[{\sf ACTIVSg10K},] same source as {\sf ACTIVSg2000} but now $n = 20\,000$, $\tilde n = 200$, $\hat n = 100$, estimation of $\|A^{-1}\|_F^2 \approx 1.3\times 10^{5}$.
 \item[{\sf CFD,}] from the SuiteSparse Matrix Collection~\cite{Davis2011} (matrix Rothberg/cfd1) and originating from the discretization of a fluid dynamics problem; $n = 70\,656$, $\tilde n = 276$, $\hat n = 256$, estimation of $\trace(A) = 70\,656$. 
 \item[{\sf CFDinv,}] identical with {\sf CFD} but estimation of $\trace(A^{-1}) \approx 6.0 \times 10^5$ instead. 
 \item[{\sf Laplace,}] matrix from second-order finite difference discretization of Poisson equation on unit square; 
 $n = 2\,500$, $\tilde n = \hat n = 50$, estimation of $\trace(A^{-1}) \approx 0.61$.
 \item[{\sf Convdiff,}] matrix from finite difference discretization of convection-diffusion equation on unit square (matrix from~\cite[Sec. 7.2]{Kressner2010} with $c_s = 1$);
 $n = 2\,500$, $\tilde n = \hat n = 50$, estimation of $\|A^{-1}\|_F^2 \approx 0.0042$.
\end{description}
In all examples for which the trace is estimated the involved matrix is symmetric positive semi-definite. The matrices {\sf Ones} and {\sf Rank-one} both have rank one but the stable rank of their (only) singular vector reshaped as a matrix is very different. The discussion in Section~\ref{sec:rademacherone} has singled out {\sf Ones} as a bad example when estimating with Rademacher vectors. {\sf CFD} is an example used in~\cite{Avron2011}. For both, {\sf Laplace} and {\sf Convdiff}, the matrix $A$ can be represented in the form
$C_1 \otimes I_{\hat n} + I_{\tilde n} \change{\otimes} C_2$ for smaller, sparse matrices $C_1,C_2$. Such matrices are of particular interest for our new rank-one estimators because the application of $A^{-1}$ (or $A^{-T}$) corresponds to the solution of a matrix Sylvester equation and such an equation can be solved much more efficiently when the right-hand side has low rank~\cite{Simoncini2016}.

\begin{figure}
\thisfloatpagestyle{empty}
\captionsetup[subfloat]{farskip=0pt,captionskip=0pt}
	\centering	\rotatebox[origin = l]{90}{\parbox{130pt}{\centering {\sf Ones}}}
	\subfloat{\includegraphics[width=0.4\textwidth]{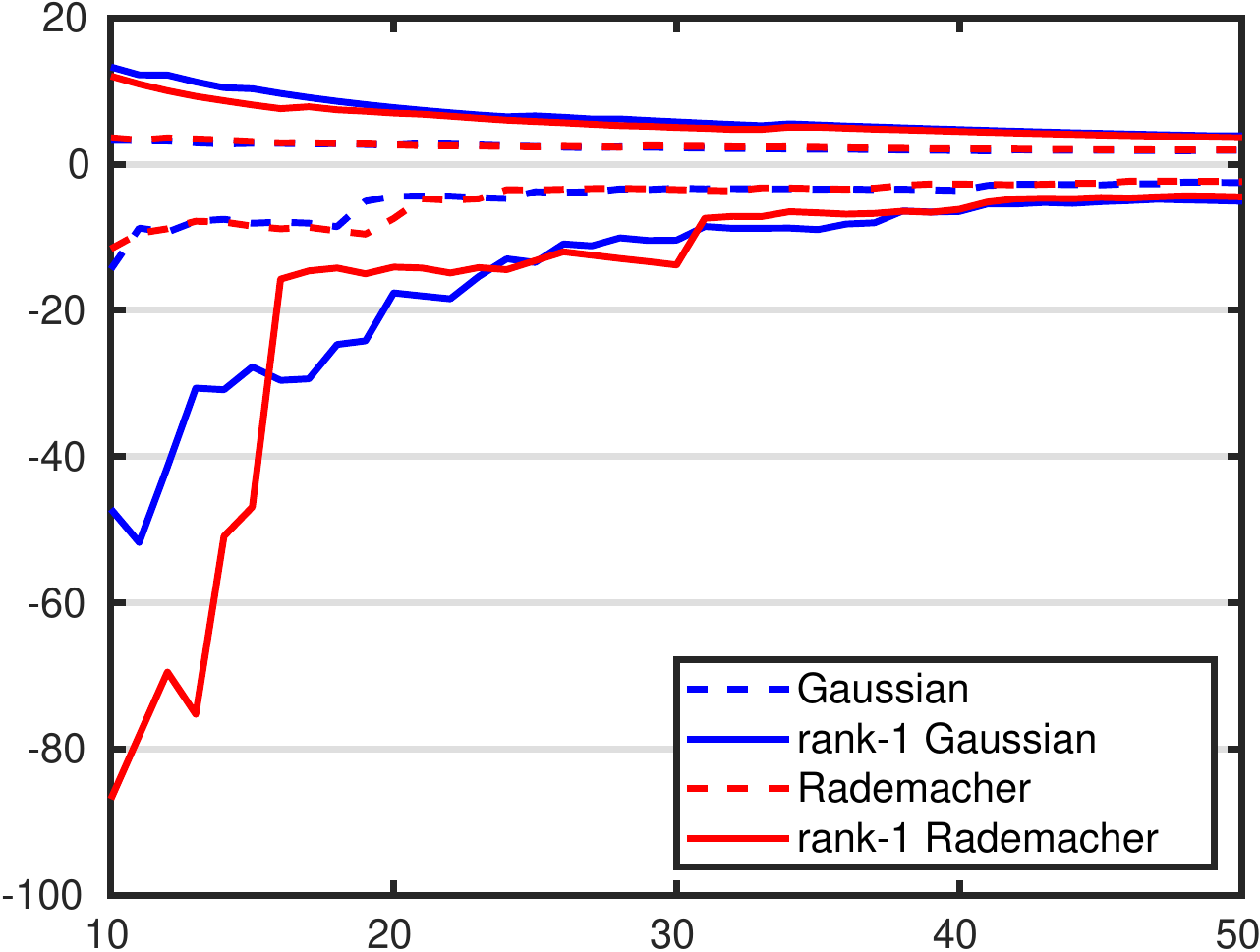}}
	\ \ 
	\rotatebox[origin = l]{90}{\parbox{130pt}{\centering {\sf Rank one}}}
	\subfloat{\includegraphics[width=0.4\textwidth]{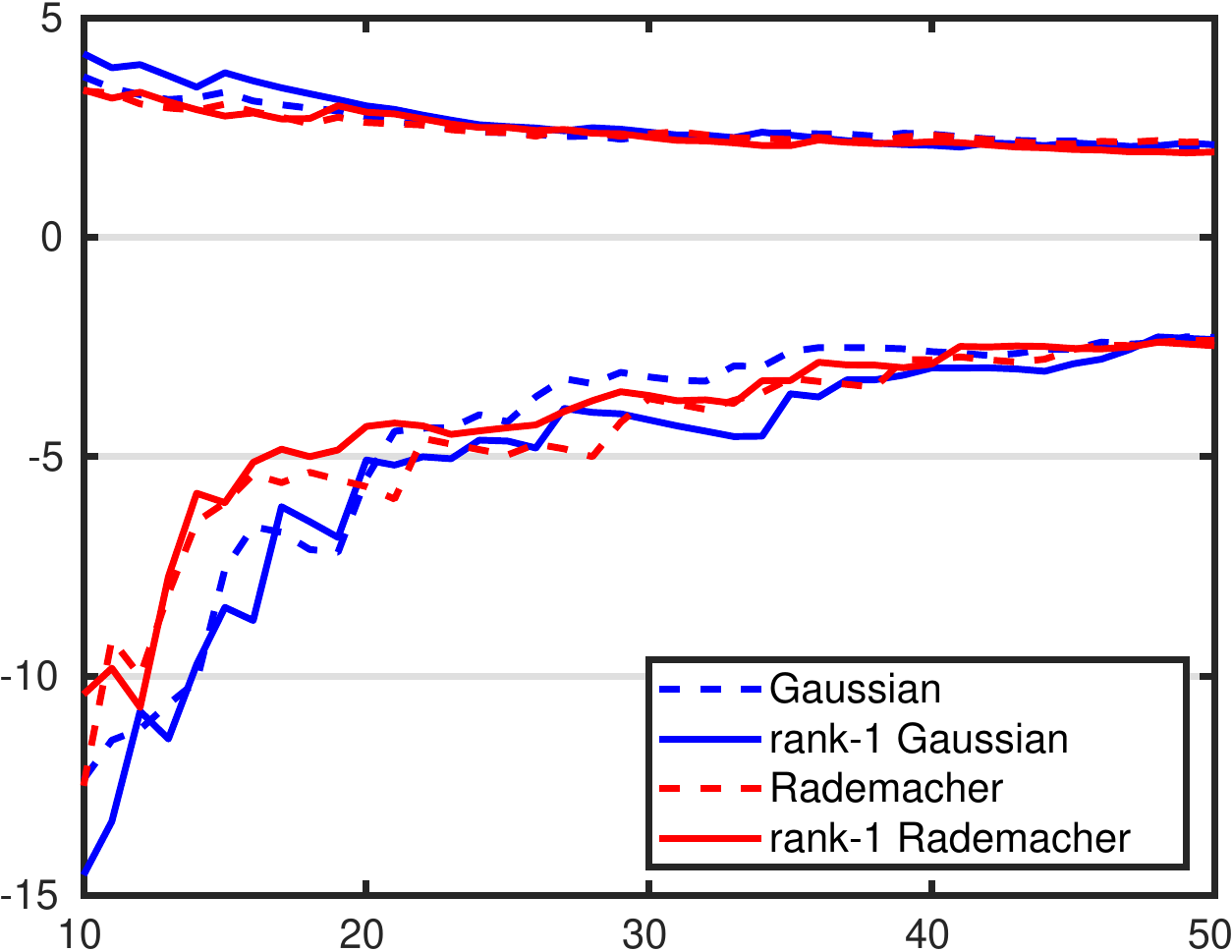}} \\ 
	\rotatebox[origin = l]{90}{\parbox{130pt}{\centering {\sf ACTIVSg2000}}}
	\subfloat{\includegraphics[width=0.4\textwidth]{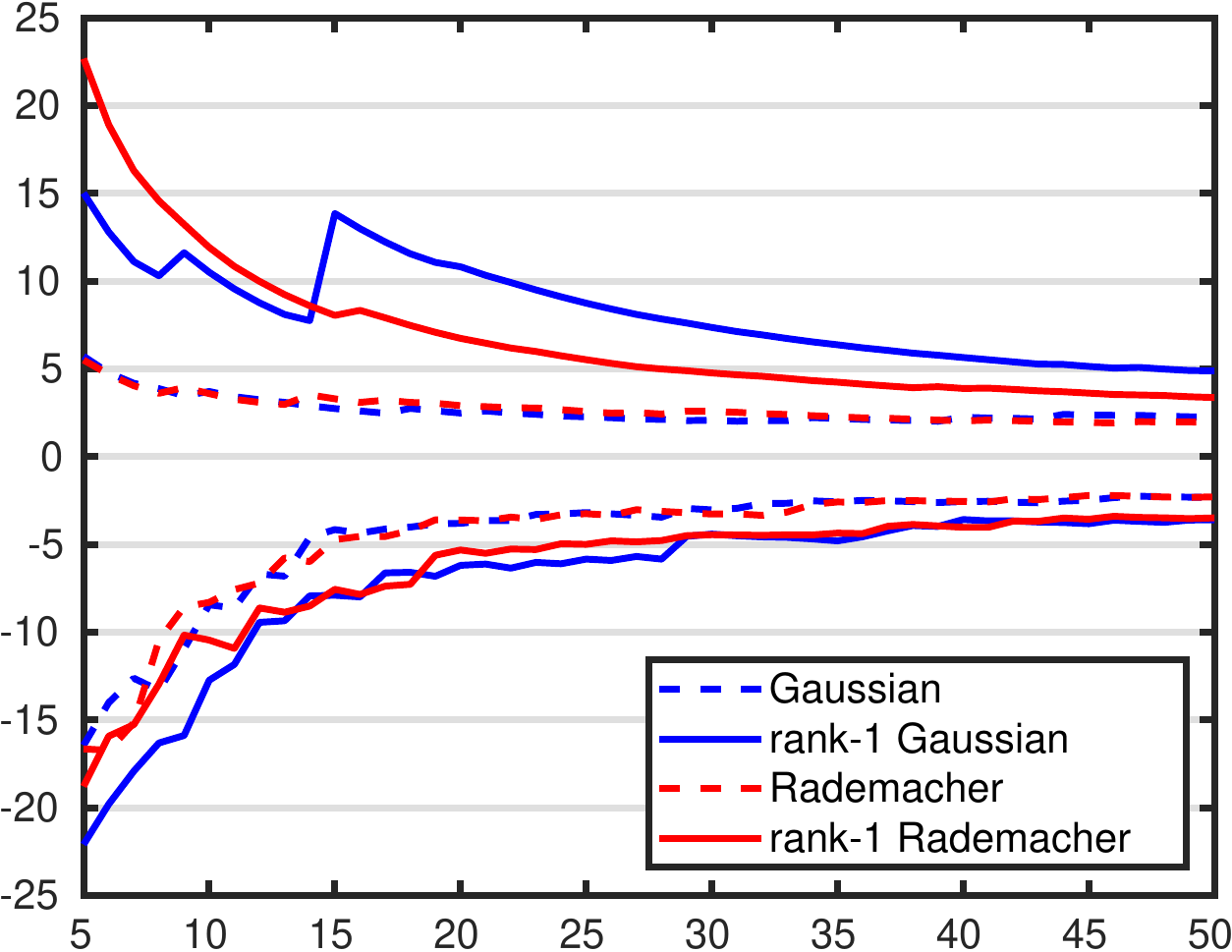}}
	\ \ 
	\rotatebox[origin = l]{90}{\parbox{130pt}{\centering {\sf ACTIVSg10K}}}
	\subfloat{\includegraphics[width=0.4\textwidth]{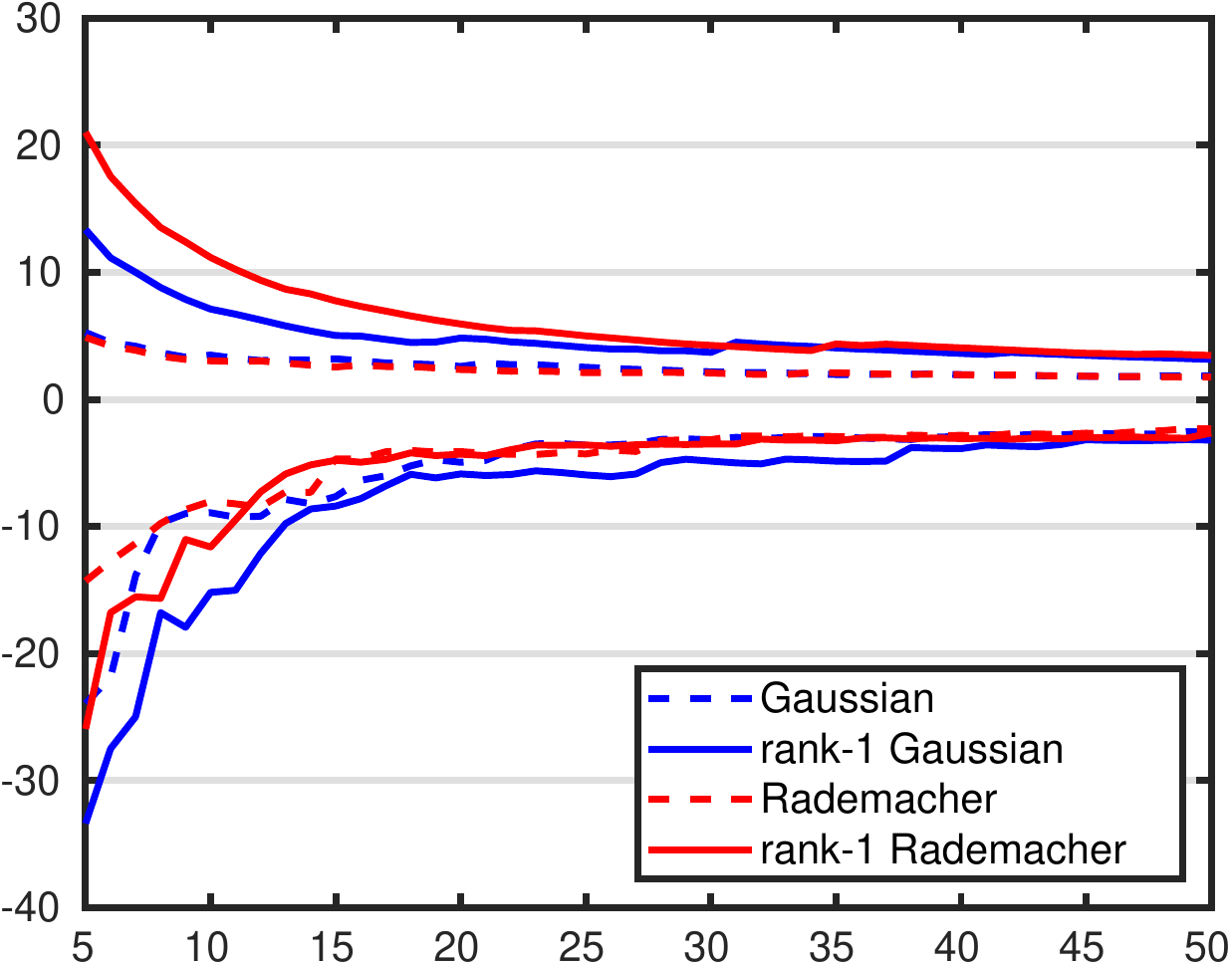}} \\
	\rotatebox[origin = l]{90}{\parbox{130pt}{\centering {\sf CFD}}}
    \subfloat{\includegraphics[width=0.4\textwidth]{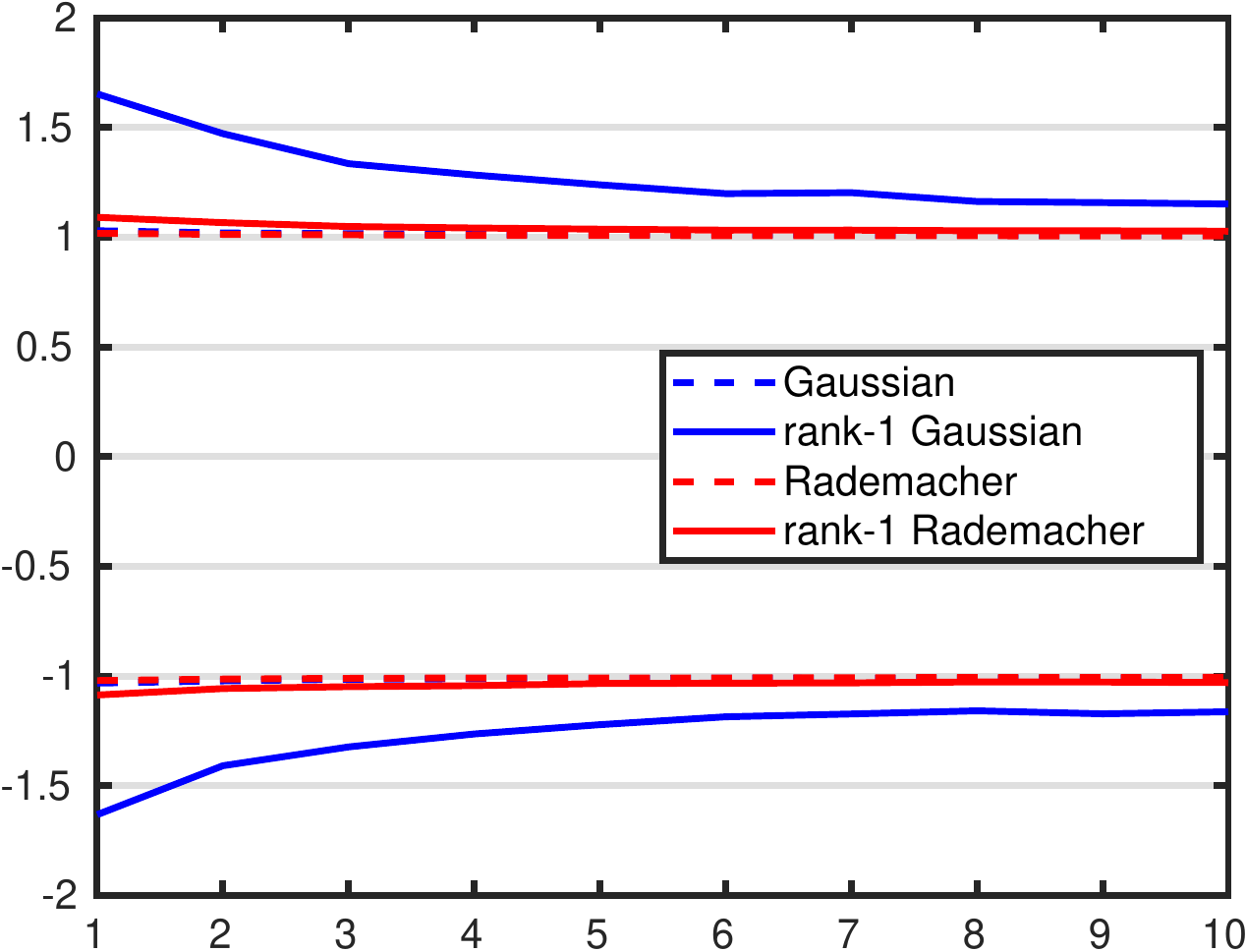}}  
    \ \ 
    \rotatebox[origin = l]{90}{\parbox{130pt}{\centering {\sf CFDinv}}}
	\subfloat{\includegraphics[width=0.4\textwidth]{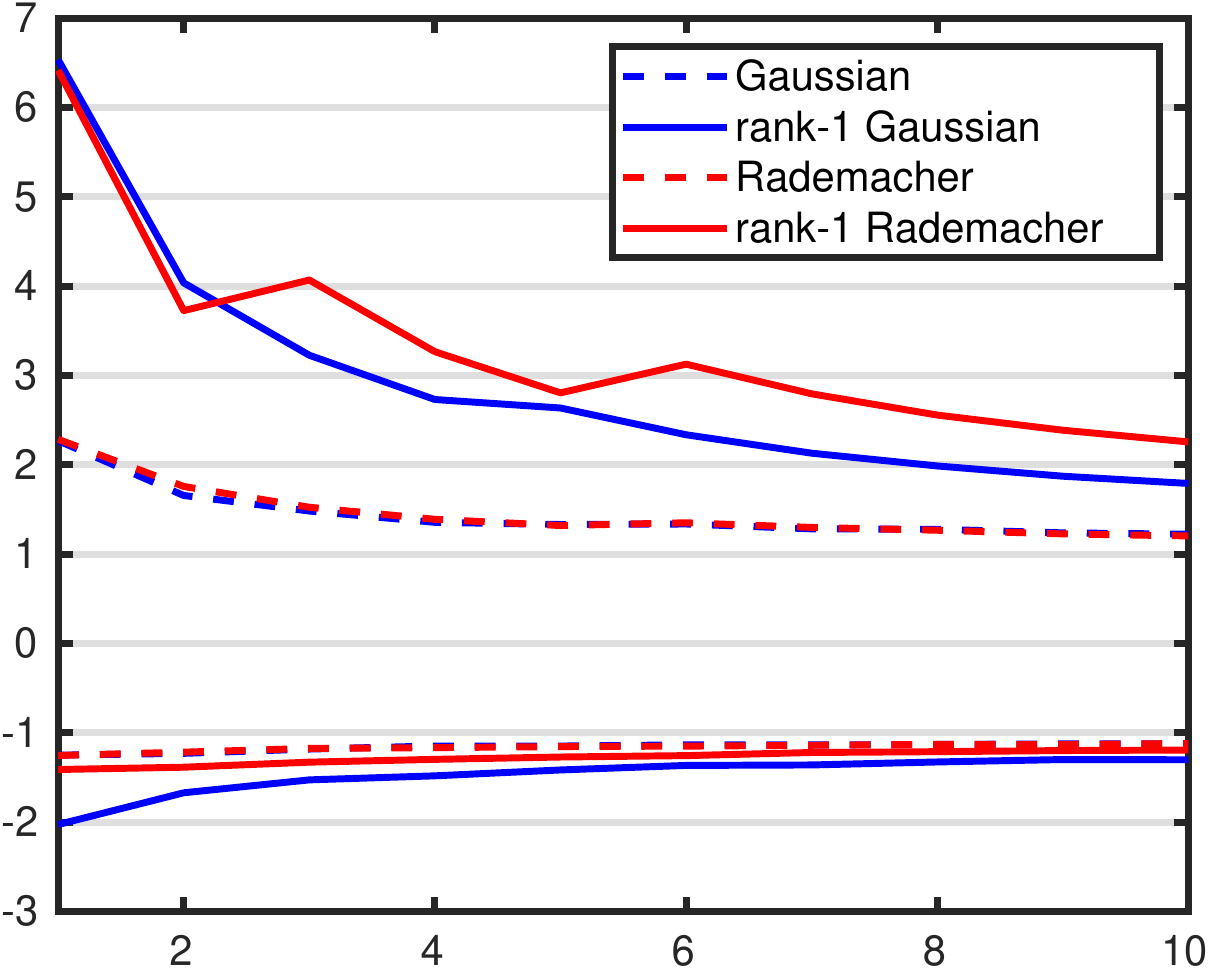}} \\ 
	\rotatebox[origin = l]{90}{\parbox{130pt}{\centering {\sf Laplace}}}
	\subfloat{\includegraphics[width=0.4\textwidth]{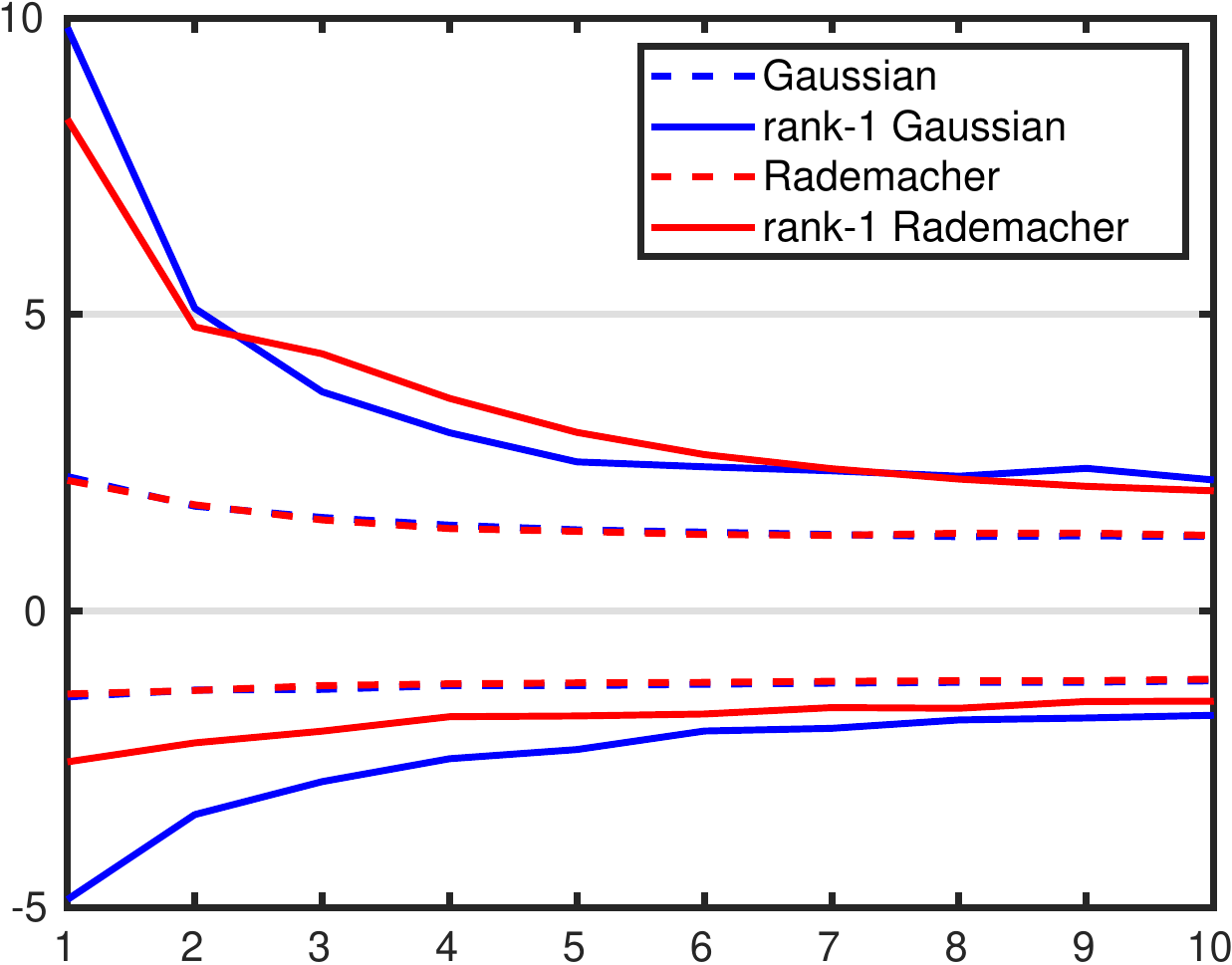}} 
	\ \ 
	\rotatebox[origin = l]{90}{\parbox{130pt}{\centering {\sf Convdiff}}} 
	\subfloat{\includegraphics[width=0.4\textwidth]{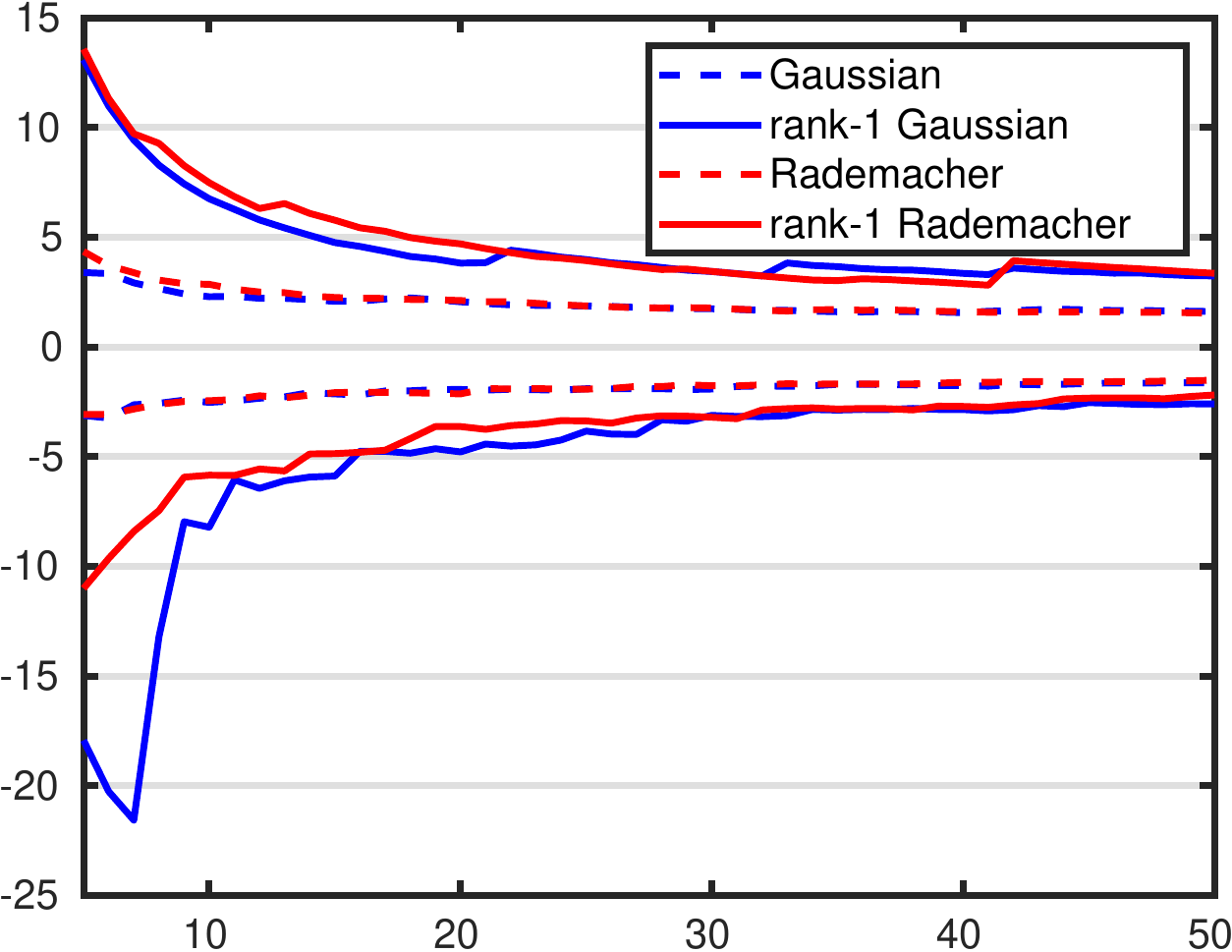}}  \\ %
\caption{Performance of trace estimators for 8 different matrices. See Section~\ref{sec:testmatrices} for details. \label{fig:testmatrices}}
\end{figure}

When the trace of $A$ itself is estimated, we used its exact trace as reference value $\mathsf{Exact}$.
When the trace or Frobenius norm of $A^{-1}$ are estimated, we used $\mathsf{Est}_{1\,000}$ with standard Gaussian random vectors as reference value $\mathsf{Exact}$. For each matrix and each type of random vector, we repeated $10\,000$ times the computation of $\mathsf{Est}_{k}$ with $k$ ranging from $1$ to up to $50$ and computed the minimum of $-\mathsf{Exact} / \mathsf{Est}_{k}$ and the maximum of $\mathsf{Est}_{k} / \mathsf{Exact}$ across all $10\,000$ runs for each $k$. The lower and upper curves in each plot of Figure~\ref{fig:testmatrices} display the minima/maxima vs. $k$ for four different types of random vectors: standard Gaussian, rank-one Gaussian, Rademacher, and rank-one Rademacher vectors. For example, for the matrix $\mathsf{Ones}$ and $k = 20$, it can be seen that all curves stay between $-20$ and $10$, which implies that for all $10\,000$ samples of $\mathsf{Est}_{k}$, the bounds
$
\mathsf{Est}_{k}/10 \le \mathsf{Exact} \le 20 \cdot \mathsf{Est}_{k}
$
were satisfied. For nearly all configurations, using rank-one random vectors instead of unstructured random vectors only has a modest impact on these worst-case under- and overestimation factors. Additional data is given in Appendix~\ref{sec:detaileddata}, which shows that the inequalities
\[
  \mathsf{Est}_{10}/30 \le \mathsf{Exact} \le 30 \cdot \mathsf{Est}_{10}
\]
are nearly always satisfied across all matrices and all types of random vectors.


\newcommand{\fre}{Fr\'{e}chet }
\newcommand{\df}[1]{Df\{#1\}}

\subsection{\fre derivative norm estimation}
\label{sec:frechet}

In this section we describe an application in which random vectors of the form $\tilde{x} \otimes \hat{x}$ are exploited to significantly speed up computation. Given a matrix function $f(A)$ \change{for} a matrix $A \in \mathbb{R}^{n \times n}$, our goal is to estimate the operator norm of the \fre derivative, which is a linear map $\df{A} : \mathbb{R}^{n \times n} \to \mathbb{R}^{n \times n}$ uniquely defined by the property $f(A+X) = f(A)+ \df{A}(X)+\mathcal{O}(\|X\|_2^2)$.
\change{This quantity measures the first-order sensitivity of the matrix function $f(A)$ under perturbation of $A$~\cite{MR2486857,MR3080997,HighamFOM}.}

It is well-known \cite[Section 3.2]{HighamFOM} that under certain conditions on the smoothness of $f$,
\begin{equation}
	\label{eq:frechet_block12}
	f \left( \mb{cc} A & X \\ 0 & A \me \right) =
		\mb{cc} f(A) & \df{A}(X) \\ 0 & f(A) \me.
\end{equation}
By vectorizing the matrices in its domain and range, the action of $\df{A}$ can be represented by an $n^2 \times n^2$ matrix $K_A$ such that $\text{vec}(\df{A}(X)) = K_A x$, where $x = \text{vec}(X)$. This yields
$$
	\|\df{A}\| \change{:= \sup_{\|X\|_F=1} \|\df{A}(X)\|_F = \sup_{\|x\|_2=1} \|K_A x\|_2 = \|K_A\|_2} = \sqrt{ \lam_{\max} (K_A^T K_A)}.
$$
Therefore, to compute $\|\df{A}\|$, one may apply the power method to the matrix $K_A^T K_A$ to compute its largest eigenvalue. The power method requires evaluating $K_A x$ and $K_A^T x$ for several vectors $x$, i.e., evaluating $\df{A}{X}$ and $\df{A^T}{X}$ for several matrices $X$, where $\text{vec}(X) = x$. This is executed via \eqref{eq:frechet_block12}: the matrix function $f$ is evaluated at a $2n \times 2n$ matrix, and the top right block is read from the resulting matrix. For larger $n$, such computation may be very demanding.

If we are only interested in an upper bound for $\|\df{A}\|$ instead of its exact value, we can apply techniques discussed in Section \ref{sec:small-sample}: let $x_j = \tilde{x}_j \otimes \hat{x}_j$, where $\tilde{x}_j$ and $\hat{x}_j$ are standard Gaussian vectors of length $n$, for $j=1, \ldots, k$. Using the maximum estimator
\begin{equation}
	\label{eq:frechet-maxest}
	\mathsf{Max}_k
		= \max_{j=1, \ldots, k} \|K_A (\tilde{x}_j \otimes \hat{x}_j)\|_2
		= \max_{j=1, \ldots, k} \|\df{A} (\hat{x}_j \tilde{x}_j^T)\|_F
\end{equation}
and applying \eqref{eq:norm2-upper-max-estimator} then guaranties the following:
$$
	\PP{ \|\df{A}\| \leq \theta \cdot \mathsf{Max}_k }
		= \PP{ \|K_A\|_2 \leq \theta \cdot \mathsf{Max}_k }
		\geq 1 - \Big( \frac{2}{\pi} \left( 2 + \ln(1 + 2\theta) \right) \theta^{-1} \Big)^k.
$$
Computing the matrix-vector products in \eqref{eq:frechet-maxest} reduces to evaluating $\df{A} (X_j)$ for rank-one matrices $X_j=\hat{x}_j \tilde{x}_j^T$. This can be done far more efficiently than evaluating $\df{A}(X)$ for general matrices $X$, by using Algorithm \ref{alg:frechet-rank-one}, slightly adapted from \cite{Kressner18Bivariate}.
Algorithm \ref{alg:frechet-rank-one} also needs to evaluate the function $f$, but for matrices of sizes at most $2\ell \times 2\ell$, where the final dimension $\ell$ of the Krylov subspace is significantly smaller than $n$. This is the source of the significant speedup when compared to the power method. Also note that only the first iteration of the power method can benefit from Algorithm \ref{alg:frechet-rank-one}, as the later iterates are generally not matrices of rank one.

\begin{algorithm2e}[h]
	\caption{Arnoldi method for approximating $\df{A}(cd^T)$}
	\label{alg:frechet-rank-one}

	\For{$\ell=1, 2, \ldots$}{%
		Perform one step of the Arnoldi method to obtain an orthonormal basis $U_{\ell}$ of the Krylov subspace $\mathcal{K}_{\ell}(A, c)$ and $G_{\ell}=U_{\ell}^T A U_{\ell}$, $\tilde{c}=U_{\ell}^T c$.

		Perform one step of the Arnoldi method to obtain an orthonormal basis $V_{\ell}$ of the Krylov subspace $\mathcal{K}_{\ell}(A^T, d)$ and $H_{\ell}=V_{\ell}^T A^T V_{\ell}$, $\tilde{d}=V_{\ell}^T d$.

		Compute $F_{\ell} = f\left( \mb{cc} G_{\ell} & \tilde{c}\tilde{d}^T \\ 0 & H_{\ell}^T \me \right)$ and set $X_{\ell} = F_{\ell}(1\!:\!{\ell}, {\ell}+1\!:\!2{\ell})$.

		\If{converged}{Stop the loop.}
	}

	Return $U_{\ell} X_{\ell} V_{\ell}^T$.
\end{algorithm2e}

To illustrate this difference, we ran both the power method and the $99.9\%$ confidence maximum estimator ($\theta=10$, $k=7$) to estimate $\|\df{A}\|$, where $f$ is the matrix exponential, and $A = -0.01( I_n \otimes T_n + T_n \otimes I_n)$. Here $T_n$ is the tridiagonal matrix with $2/(n-1)^2$ on the main diagonal and $-1/(n-1)^2$ on the first upper and lower suddiagonal, i.e., the 1D discrete Laplacian on $[0, 1]$. Running in Matlab R2019b on an Intel i5 4690K processor, we obtain the following results:

\begin{center}
	\begin{tabular}{|c||c|c||c|c|c|}
		\hline
		$n$ & time(power method) & $\|\df{A}\|$ & time($\mathsf{Max}_k$) & $\theta \cdot \mathsf{Max}_k$ & maximum $\ell$\\ \hline
		 10 &    0.12            &    0.86   &     0.05    &    147.10   &  20 \\
		 20 &    1.81            &    0.85   &     0.10    &    151.00   &  30 \\
		 30 &   16.49            &    0.78   &     0.38    &    115.44   &  45 \\
		 40 &   74.06            &    0.80   &     1.18    &    102.13   &  55 \\
		 50 &  275.52            &    0.82   &     2.95    &     93.38   &  70 \\ \hline
	\end{tabular}
\end{center}

Note that $A$ is an $n^2 \times n^2$ matrix.
We ran $7$ iterations of the power method; the third column shows the approximation of $\|\df{A}\|$ it reported. The fifth column shows the $99.9\%$ confidence upper bound for $\|\df{A}\|$ as reported by the maximum estimator. In the last column is the maximum dimension $\ell$ of all Krylov subspaces needed for the computation of $\mathsf{Max}_k$. We stop the Arnoldi iteration once $\|F_\ell - \smb F_{\ell-1} & 0 \\ 0 & 0 \sme\| < 10^{-8}$, as suggested in \cite[Section 2.3]{BecKreSch17}. All times are given in seconds. While the upper bounds provided by the maximum estimators are, in this case, about $100$ times larger than the actual norm of the \fre derivative, this may be sufficient as a rough estimate. Such an estimate can be rapidly computed by using random vectors studied in this paper.

\section{Conclusions}

In this work we have provided theoretical and experimental evidence that rank-one random vectors are suited for norm  and trace estimation. While their performance is consistently worse compared to unstructured vectors, this can be easily mitigated by, e.g., increasing the constants or increasing the number of samples in the stochastic trace estimator.

It is tempting to ask whether the results of this paper have a meaningful extension to higher-order tensors, \change{that is, norm and trace estimation with Kronecker products of $d$ vectors.} 
\change{Without further assumptions on $A$, the techniques used in this work will lead to estimates of the success probability that vanish exponentially fast as $d$ increases. This is explicit in the results from~\cite{Ver19} but also the Khinchine inequalities~\cite{Latala2006} and, more generally, moments of order-$d$ chaos exhibit exponential growth.}

\begin{paragraph}{Acknowledgments.}
 This manuscript concludes a long journey and we gratefully acknowledge the numerous discussions with colleagues on the aims and techniques of this work, including Robert Dalang, Hrvoje Planini\'{c}, Holger Rauhut, and Andr\'e Uschmajew. \change{We also thank the referees for their careful reading and constructive remarks.}
 
 This work was supported by the  SNSF research project \emph{Low-rank
updates of matrix functions and fast eigenvalue solvers}, the Croatian Science Foundation under the grant HRZZ-6268 - \emph{Randomized low rank algorithms and applications to parameter dependent problems},
 and the COST action CA18232 - \emph{Mathematical models for interacting dynamics on networks}.
\end{paragraph}

\bibliographystyle{plain}
\bibliography{refs}

\appendix

\section{Appendix}

The following lemma provides a Chernoff bound for chi-square distributions. This result is certainly well known; we include it for completeness.
\begin{lemma}
	\label{lem:chi_bound}
	Let $X$ be a random variable having a chi-square distribution with $k$ degrees of freedom. For $\theta > 1$, it holds that
	$$
		\PP{X > k\theta} \leq (\theta e^{1-\theta})^{k/2}.
	$$
\end{lemma}
\begin{proof}
	For $t>0$, we let $M_X(t) = (1-2t)^{-k/2}$ denote the moment generating function of $X$. By the Markov inequality,
	\begin{align*}
		\PP{X > k\theta}
			&= \PP{tX/k > t\theta}
			= \PP{e^{tX/k} > e^{t\theta}} \\
			&\leq \EE[ e^{tX/k} ] \cdot e^{-t \theta}
			= M_{X}(t/k) e^{-t \theta}
			= (1 - 2t/k)^{-k/2} e^{-t \theta},
	\end{align*}
	which holds for $t/k < 1/2$. Choosing $t = k/2 - k/{2\theta}$ implies
	$$
		\PP{X > k\theta}
			\leq \theta^{k/2} e^{k/2(1 - \theta)}
			= (\theta e^{1-\theta})^{k/2}.
	$$
\end{proof}

The following two results on the moments and the moment generating function of decoupled second-order Gaussian chaos are closely related to  existing results by Lata\l{}a~\cite{Latala2006}; see also the monograph~\cite{Boucheron2013}. We include these results for the convenience of the reader.
\begin{lemma} \label{lemma:gaussianchaos}
	Let $Q \in \R^{\hat n \times \tilde n}$, and let $\hat{x} \sim \normal(0, I_{\hat n})$, $\tilde{x} \sim \normal(0, I_{\tilde n})$ . For $Z = \hat x^T Q \tilde x$ it holds that
	\[
	 \EE[Z^2] = \|Q\|_F, \qquad \EE[Z^4] = 3(2 \|Q\|_{(4)}^4 + \|Q\|_F^4) \le 9\|Q\|_F^4,
	\]
	where $\|\cdot \|_{(4)}$ denotes the Schatten-$4$ norm~\cite[Sec. 7.4]{Horn2013} of a matrix.
	For any even $k$, we have
	\begin{equation} \label{eq:momentboundsgauss}
	 \EE[Z^k] \le \big( (k-1)!! \big)^2 \|Q\|_F^k,
    \end{equation}
	where $(k-1)!! = (k-1)(k-3) \cdots 3 \cdot 1$ denotes the double factorial. For odd $k$, $\EE[Z^k] = 0$.

\end{lemma}
\begin{proof}
For the second moment, we obtain
\[
 \EE[Z^2] = \EE_{\hat x}\big[ \EE_{\tilde x} \big[ (\hat x^T Q \tilde x)^2 \vert \hat x \big] \big] =
 \EE_{\hat x}\big[ \|Q^T \hat x\|^2_2 \big]  = \|Q\|_F^2,
\]
where the second equality follows from the fact that for $y = Q^T \hat x$ with fixed $\hat x$, the random variable $y^T \tilde x$ is normal with zero mean and variance $\|y\|_2^2$. Noting that the fourth moment of such a normal random variable is $3\|y\|_2^4$,
an analogous argument shows $\EE[Z^4] = 3 \cdot \EE\big[ \|Q^T \hat x\|^4_2 \big]$. To proceed from here, we may
assume -- without loss of generality -- that $\hat n \le \tilde n$ and that $Q$ is a diagonal matrix with the singular values $\sigma_1\ge \cdots \ge \sigma_{\hat n}\ge 0$ on the diagonal; see, e.g., the proof of Theorem~\ref{tm:norm2-left-tail}. This gives
\begin{align*}
 \EE\big[ \|Q^T \hat x\|^4_2 \big] &= \EE\big[ \big( \sigma_1^2 \hat x_1^2 + \cdots + \sigma_{\hat n}^2 \hat x_{\hat n}^2\big)^2 \big] = \sum_{ij} \sigma_i^2 \sigma_j^2  \EE[\hat x_i^2\hat x_j^2] \\
 &= \sum_i \sigma_i^4 \EE[\hat x_i^4] +  \sum_{i \not= j} \sigma_i^2 \sigma_j^2  \EE[\hat x_i^2] \cdot \EE[\hat x_j^2]
 = 3 \sum_i \sigma_i^4 + \sum_{i \not= j} \sigma_i^2 \sigma_j^2 \\
 &= 2 \sum_i \sigma_i^4 + \sum_{ij} \sigma_i^2 \sigma_j^2 = 2 \|Q\|_{(4)}^4 + \|Q\|_F^4,
\end{align*}
which establishes the claimed expression for $\EE[Z^4]$. The upper bound $9\|Q\|_F^4$ follows from $\|Q\|_{(4)} \le \|Q\|_F$.

For general even $k$, the statement and proof of~\eqref{eq:momentboundsgauss} is contained in the proof of Lemma 7.1 in~\cite{Cai2015}. The proof that follows is slightly simpler. We first note that the $k$th moment of a centered normal random variable with variance $\sigma^2$ is given by
$(k-1)!! \sigma^k$. In turn, $\EE[Z^k] = (k-1)!!\cdot \EE\big[ \|Q^T \hat x\|^k_2 \big]$. We proceed as above and obtain
\[
 \EE\big[ \|Q^T \hat x\|^k_2 \big]  = \sum_{i_1,\ldots,i_{k/2}} \sigma^2_{i_1} \cdots \sigma^2_{i_{k/2}} \EE[ \hat x_{i_1}^2 \cdots \hat x_{k/2}^2].
\]
Using that $\EE[ \hat x_{i}^{2p} \hat x_{j}^{2q}] = (2p-1)!! \cdot (2q-1)!! \le (2(p+q)-1)!! = \EE[ \hat x_{i}^{2(p+q)}]$ for any $p,q\in \mathbb N$ and $i\not=j$, we obtain
\[
\EE\big[ \|Q^T \hat x\|^k_2 \big] \le (k-1)!! \sum_{i_1,\ldots,i_{k/2}} \sigma^2_{i_1} \cdots \sigma^2_{i_{k/2}} = (k-1)!! \cdot \|Q\|_F^k,
\]
which concludes the proof of~\eqref{eq:momentboundsgauss}.

The statement on odd $k$ follows from the symmetry of the distribution: $Z$ and $-Z = \hat x^T Q (-\tilde x)$ have the same distribution and hence $\EE\big[Z^k] =
\EE\big[(-Z)^k] = \EE\big[{-Z^k}] = -\EE\big[Z^k]$. This shows $\EE\big[Z^k] = 0$.
\end{proof}

\begin{corollary} \label{corollary:gausmgf}
For $Z$ as in Lemma~\ref{lemma:gaussianchaos} with $\|Q\|_F  = 1$, the moment generating function exists and is bounded by
\[
 \EE\big[ \exp(tZ) \big] \le \frac{1}{\sqrt{1-t^2}}
\]
for $|t| < 1$.
\end{corollary}
\begin{proof}
Using the result of Lemma~\ref{lemma:gaussianchaos}, it follows that
\begin{align*}
 \EE\big[ \exp(tZ) \big] & = \sum_{k = 0}^\infty \frac{1}{k!} \EE\big[ Z^k \big] t^k = \sum_{k = 0}^\infty \frac{1}{(2k)!} \EE\big[ Z^{2k} \big] t^{2k} \\
 &\le \sum_{k = 0}^\infty \frac{\big( (k-1)!! \big)^2}{(2k)!}  t^{2k}
 = \sum_{k = 0}^\infty \frac{1}{4^k} {2k \choose k}  t^{2k} = \frac{1}{\sqrt{1-t^2}},
\end{align*}
where the last step follows from Taylor expansion.
\end{proof}

\pagebreak

\section{Detailed data on performance of estimators} \label{sec:detaileddata}

The following tables provide detailed data on the performance of trace/Frobenius norm estimation with the stochastic trace estimator using rank-one/unstructured Gaussian/Rademacher vectors.
\begin{center}
Legend:
{\sf G}$=$Gaussian,
{\sf G1}$=$rank-one Gaussian,
{\sf R}$=$Rademacher,
{\sf R1}$=$rank-one Rademacher.
\end{center}
For each value of $\theta$ and $k$ the upper value shows the ratio of events that the estimator times $\theta$ does not provide an upper bound and 
the lower value shows the ratio of events that the estimator divided by $\theta$ does not provide a lower bound. For example, when $A$ is the matrix of all ones,
the inequality $\trace(A) \le 8\cdot \mathsf{Est}_5$ fails for only $33$ out of $10\,000$ events
while 
the inequality $\trace(A) \ge 1/8\cdot \mathsf{Est}_5$ fails for $1\,201$ out of $10\,000$ events when using rank-one Gaussian vectors.\\[0.3cm]%
\begin{tabular}{cc} 
\!\!\!\!\!\!\!\!Trace of matrix of all ones & Trace of $vv^T$ with vectorized identity matrix $v$ \\
 \small 
\!\!\!\!\!\!\!\!\begin{tabular}{ccrrrrrr} \hline
& & \multicolumn{1}{c}{\!\!$\theta = 1.2$\!\!} & \multicolumn{1}{c}{$\theta = 2$} & \multicolumn{1}{c}{$\theta = 4$} & \multicolumn{1}{c}{$\theta = 8$} & \multicolumn{1}{c}{$\theta = 30$} \\ \hline
 \multirow{8}{*}{\rotatebox{90}{$k = 1$}} &
 \multirow{2}{*}{\sf G}
    &$0.2758$    &$0.1619$    &$0.0467$    &$0.0053$    &$0.0000$ \\
   &&$0.6376$    &$0.5187$    &$0.3867$    &$0.2796$    &$0.1504$ \\  \cline{2-7}
 & \multirow{2}{*}{\sf G1}
    &$0.1814$    &$0.1244$    &$0.0637$    &$0.0248$    &$0.0016$ \\
   &&$0.7683$    &$0.6959$    &$0.5919$    &$0.4937$    &$0.3332$ \\ \cline{2-7}
 & \multirow{2}{*}{\sf R}
    &$0.2726$    &$0.1587$    &$0.0423$    &$0.0058$    &$0.0000$ \\
   &&$0.6282$    &$0.5136$    &$0.3791$    &$0.2644$    &$0.1361$ \\ \cline{2-7}
 & \multirow{2}{*}{\sf R1}
    &$0.1952$    &$0.1332$    &$0.0521$    &$0.0210$    &$0.0014$ \\
   &&$0.7401$    &$0.6759$    &$0.6100$    &$0.4950$    &$0.3406$ \\ \hline
 \multirow{8}{*}{\rotatebox{90}{$k = 5$}} &
 \multirow{2}{*}{\sf G}
    &$0.3063$    &$0.0728$    &$0.0011$    &$0.0000$    &$0.0000$ \\
   &&$0.4742$    &$0.2248$    &$0.0606$    &$0.0143$    &$0.0005$ \\ \cline{2-7}
 & \multirow{2}{*}{\sf G1}
    &$0.2594$    &$0.1264$    &$0.0303$    &$0.0033$    &$0.0000$ \\
   &&$0.6236$    &$0.4493$    &$0.2495$    &$0.1201$    &$0.0196$ \\ \cline{2-7}
 & \multirow{2}{*}{\sf R}
    &$0.3029$    &$0.0742$    &$0.0013$    &$0.0000$    &$0.0000$ \\
   &&$0.4733$    &$0.2262$    &$0.0605$    &$0.0124$    &$0.0006$ \\ \cline{2-7}
 & \multirow{2}{*}{\sf R1}
    &$0.2667$    &$0.1294$    &$0.0303$    &$0.0039$    &$0.0000$ \\
   &&$0.6114$    &$0.4367$    &$0.2357$    &$0.1061$    &$0.0193$ \\ \hline
 \multirow{8}{*}{\rotatebox{90}{$k = 10$}} &
 \multirow{2}{*}{\sf G}
    &$0.2912$    &$0.0275$    &$0.0000$    &$0.0000$    &$0.0000$ \\
   &&$0.3982$    &$0.1079$    &$0.0078$    &$0.0001$    &$0.0000$ \\ \cline{2-7}
 & \multirow{2}{*}{\sf G1}
    &$0.2807$    &$0.1053$    &$0.0155$    &$0.0009$    &$0.0000$ \\
   &&$0.5481$    &$0.3153$    &$0.1079$    &$0.0271$    &$0.0009$ \\ \cline{2-7}
 & \multirow{2}{*}{\sf R}
    &$0.2873$    &$0.0298$    &$0.0000$    &$0.0000$    &$0.0000$ \\
   &&$0.3930$    &$0.1069$    &$0.0117$    &$0.0004$    &$0.0000$ \\ \cline{2-7}
 & \multirow{2}{*}{\sf R1}
    &$0.2749$    &$0.1024$    &$0.0121$    &$0.0008$    &$0.0000$ \\
   &&$0.5507$    &$0.3142$    &$0.1085$    &$0.0240$    &$0.0005$ \\ \hline
\end{tabular} \!\!&\!\!  \small 
\begin{tabular}{ccrrrrrr} \hline
& & \multicolumn{1}{c}{\!\!$\theta = 1.2$\!\!} & \multicolumn{1}{c}{$\theta = 2$} & \multicolumn{1}{c}{$\theta = 4$} & \multicolumn{1}{c}{$\theta = 8$} & \multicolumn{1}{c}{$\theta = 30$} \\ \hline
 \multirow{8}{*}{\rotatebox{90}{$k = 1$}} &
 \multirow{2}{*}{\sf G}
    &$0.2810$    &$0.1598$    &$0.0455$    &$0.0045$    &$0.0000$ \\
   &&$0.6266$    &$0.5083$    &$0.3736$    &$0.2676$    &$0.1334$ \\  \cline{2-7}
 & \multirow{2}{*}{\sf G1}
    &$0.2729$    &$0.1570$    &$0.0490$    &$0.0058$   &$ 0.0000$ \\
   &&$0.6411$    &$0.5220$    &$0.3869$    &$0.2780$   &$ 0.1456$ \\ \cline{2-7}
 & \multirow{2}{*}{\sf R}
    &$0.3249$    &$0.1244$    &$0.0342$    &$0.0040$   &$ 0.0000$ \\
   &&$0.6751$    &$0.5124$    &$0.3190$    &$0.3190$   &$ 0.1102$ \\ \cline{2-7}
 & \multirow{2}{*}{\sf R1}
    &$0.3256$    &$0.1196$    &$0.0318$    &$0.0020$    &$0.0000$ \\
   &&$0.6744$    &$0.5137$    &$0.3250$    &$0.3250$    &$0.1115$ \\ \hline
 \multirow{8}{*}{\rotatebox{90}{$k = 5$}} &
 \multirow{2}{*}{\sf G}
    &$0.3122$    &$0.0795$    &$0.0017$    &$0.0000$    &$0.0000$ \\
   &&$0.4669$    &$0.2215$    &$0.0578$    &$0.0115$    &$0.0006$ \\ \cline{2-7}
 & \multirow{2}{*}{\sf G1}
    &$0.2992$    &$0.0781$    &$0.0016$    &$0.0000$    &$0.0000$ \\
   &&$0.4914$    &$0.2317$    &$0.0670$    &$0.0150$    &$0.0006$ \\ \cline{2-7}
 & \multirow{2}{*}{\sf R}
    &$0.2995$    &$0.0725$    &$0.0008$    &$0.0000$    &$0.0000$ \\
   &&$0.4756$    &$0.2257$    &$0.0603$    &$0.0120$    &$0.0008$ \\ \cline{2-7}
 & \multirow{2}{*}{\sf R1}
    &$0.2973$    &$0.0750$    &$0.0012$    &$0.0000$    &$0.0000$ \\
   &&$0.4847$    &$0.2361$    &$0.0666$    &$0.0166$    &$0.0012$ \\ \hline
 \multirow{8}{*}{\rotatebox{90}{$k = 10$}} &
 \multirow{2}{*}{\sf G}
    &$0.2779$    &$0.0285$    &$0.0000$    &$0.0000$    &$0.0000$ \\
   &&$0.4111$    &$0.1103$    &$0.0110$    &$0.0006$    &$0.0000$ \\ \cline{2-7}
 & \multirow{2}{*}{\sf G1}
    &$0.2872$    &$0.0332$    &$0.0000$    &$0.0000$    &$0.0000$ \\
   &&$0.4137$    &$0.1141$    &$0.0091$    &$0.0006$    &$0.0000$ \\ \cline{2-7}
 & \multirow{2}{*}{\sf R}
    &$0.2800$    &$0.0289$    &$0.0000$    &$0.0000$    &$0.0000$ \\
   &&$0.4069$    &$0.1072$    &$0.0089$    &$0.0005$    &$0.0000$ \\ \cline{2-7}
 & \multirow{2}{*}{\sf R1}
    &$0.2771$    &$0.0269$    &$0.0000$    &$0.0000$    &$0.0000$ \\
   &&$0.4144$    &$0.1003$    &$0.0097$    &$0.0004$    &$0.0000$ \\ \hline
\end{tabular}
\end{tabular}
$\ $ \\%
\begin{tabular}{cc} 
\!\!\!\!\!\!\!\!Frobenius norm of inverse of ACTIVSg2000 & Frobenius norm of inverse of ACTIVSg10K \\
 \small 
\!\!\!\!\!\!\!\!
\begin{tabular}{ccrrrrrr} \hline
& & \multicolumn{1}{c}{\!\!$\theta = 1.2$\!\!} & \multicolumn{1}{c}{$\theta = 2$} & \multicolumn{1}{c}{$\theta = 4$} & \multicolumn{1}{c}{$\theta = 8$} & \multicolumn{1}{c}{$\theta = 30$} \\ \hline
 \multirow{8}{*}{\rotatebox{90}{$k = 1$}} &
 \multirow{2}{*}{\sf G}
   &$0.2619$   &$0.1439$   &$0.0396$   &$0.0041$   &$0.0000$\\
&  &$0.6451$   &$0.5144$   &$0.3550$   &$0.2284$   &$0.0017$\\  \cline{2-7}
 & \multirow{2}{*}{\sf G1}
   &$0.1977$   &$0.1264$   &$0.0554$   &$0.0192$   &$0.0008$\\
&  &$0.7399$   &$0.6385$   &$0.4919$   &$0.3303$   &$0.0367$\\ \cline{2-7}
 & \multirow{2}{*}{\sf R}
   &$0.2687$   &$0.1507$   &$0.0390$   &$0.0028$   &$0.0000$\\
&  &$0.6352$   &$0.5044$   &$0.3502$   &$0.2189$   &$0.0004$\\ \cline{2-7}
 & \multirow{2}{*}{\sf R1}
   &$0.1922$   &$0.1202$   &$0.0519$   &$0.0159$   &$0.0002$\\
&  &$0.7370$   &$0.6151$   &$0.4414$   &$0.2808$   &$0.0007$\\ \hline
 \multirow{8}{*}{\rotatebox{90}{$k = 5$}} &
 \multirow{2}{*}{\sf G}
   &$0.2978$   &$0.0637$   &$0.0006$   &$0.0000$   &$0.0000$\\
&  &$0.4724$   &$0.2079$   &$0.0430$   &$0.0039$   &$0.0000$\\ \cline{2-7}
 & \multirow{2}{*}{\sf G1}
   &$0.2521$   &$0.1113$   &$0.0210$   &$0.0014$   &$0.0000$\\
&  &$0.6027$   &$0.3786$   &$0.1405$   &$0.0309$   &$0.0000$\\ \cline{2-7}
 & \multirow{2}{*}{\sf R}
   &$0.2949$   &$0.0635$   &$0.0005$   &$0.0000$   &$0.0000$\\
&  &$0.4668$   &$0.2064$   &$0.0426$   &$0.0036$   &$0.0000$\\ \cline{2-7}
 & \multirow{2}{*}{\sf R1}
   &$0.2503$   &$0.1025$   &$0.0183$   &$0.0013$   &$0.0000$\\
&  &$0.5962$   &$0.3607$   &$0.1124$   &$0.0149$   &$0.0000$\\ \hline
 \multirow{8}{*}{\rotatebox{90}{$k = 10$}} &
 \multirow{2}{*}{\sf G}
   &$0.2702$   &$0.0229$   &$0.0000$   &$0.0000$   &$0.0000$\\
&  &$0.4067$   &$0.0991$   &$0.0050$   &$0.0002$   &$0.0000$\\  \cline{2-7}
 & \multirow{2}{*}{\sf G1}
   &$0.2694$   &$0.0816$   &$0.0074$   &$0.0003$   &$0.0000$\\  
&  &$0.5263$   &$0.2398$   &$0.0415$   &$0.0018$   &$0.0000$\\ \cline{2-7}
 & \multirow{2}{*}{\sf R}
   &$0.2702$   &$0.0234$   &$0.0000$   &$0.0000$   &$0.0000$\\
&  &$0.4063$   &$0.0946$   &$0.0052$   &$0.0001$   &$0.0000$\\ \cline{2-7}
 & \multirow{2}{*}{\sf R1}
   &$0.2654$   &$0.0775$   &$0.0065$   &$0.0001$   &$0.0000$\\
&  &$0.5244$   &$0.2244$   &$0.0285$   &$0.0012$   &$0.0000$\\ \hline
\end{tabular}
\!\!&\!\!  \small \begin{tabular}{ccrrrrrr} \hline
& & \multicolumn{1}{c}{\!\!$\theta = 1.2$\!\!} & \multicolumn{1}{c}{$\theta = 2$} & \multicolumn{1}{c}{$\theta = 4$} & \multicolumn{1}{c}{$\theta = 8$} & \multicolumn{1}{c}{$\theta = 30$} \\ \hline
 \multirow{8}{*}{\rotatebox{90}{$k = 1$}} &
 \multirow{2}{*}{\sf G}
    &$0.2717$   &$0.1507$   &$0.0394$   &$0.0036$   &$0.0000$ \\
   &&$0.6389$   &$0.5135$   &$0.3560$   &$0.2129$   &$0.0195$ \\  \cline{2-7}
 & \multirow{2}{*}{\sf G1}
    &$0.2108$   &$0.1362$   &$0.0583$   &$0.0180$   &$0.0003$ \\
   &&$0.7271$   &$0.6289$   &$0.4922$   &$0.3469$   &$0.1060$ \\ \cline{2-7}
 & \multirow{2}{*}{\sf R}
    &$0.2750$   &$0.1480$   &$0.0360$   &$0.0027$   &$0.0000$ \\
   &&$0.6262$   &$0.4961$   &$0.3356$   &$0.1929$   &$0.0167$ \\ \cline{2-7}
 & \multirow{2}{*}{\sf R1}
    &$0.2135$   &$0.1231$   &$0.0435$   &$0.0122$   &$0.0001$ \\
   &&$0.6937$   &$0.5324$   &$0.3667$   &$0.2369$   &$0.0468$ \\ \hline
 \multirow{8}{*}{\rotatebox{90}{$k = 5$}} &
 \multirow{2}{*}{\sf G}
    &$0.2848$   &$0.0641$   &$0.0006$   &$0.0000$   &$0.0000$ \\
   &&$0.4787$   &$0.2116$   &$0.0415$   &$0.0038$   &$0.0000$ \\ \cline{2-7}
 & \multirow{2}{*}{\sf G1}
    &$0.2678$   &$0.1053$   &$0.0171$   &$0.0012$   &$0.0000$  \\
   &&$0.5802$   &$0.3629$   &$0.1340$   &$0.0313$   &$0.0001$ \\ \cline{2-7}
 & \multirow{2}{*}{\sf R}
    &$0.2869$   &$0.0538$   &$0.0002$   &$0.0000$   &$0.0000$ \\
   &&$0.4678$   &$0.1963$   &$0.0382$   &$0.0036$   &$0.0000$ \\ \cline{2-7}
 & \multirow{2}{*}{\sf R1}
    &$0.2615$   &$0.0947$   &$0.0118$   &$0.0011$   &$0.0000$ \\
   &&$0.5599$   &$0.2825$   &$0.0531$   &$0.0066$   &$0.0000$ \\ \hline
 \multirow{8}{*}{\rotatebox{90}{$k = 10$}} &
 \multirow{2}{*}{\sf G}
    &$0.2532$   &$0.0187$   &$0.0000$   &$0.0000$   &$0.0000$ \\
   &&$0.4246$   &$0.1004$   &$0.0035$   &$0.0001$   &$0.0000$ \\ \cline{2-7}
 & \multirow{2}{*}{\sf G1}
    &$0.2680$   &$0.0690$   &$0.0047$   &$0.0001$   &$0.0000$ \\
   &&$0.5135$   &$0.2339$   &$0.0400$   &$0.0027$   &$0.0000$ \\ \cline{2-7}
 & \multirow{2}{*}{\sf R}
    &$0.2514$   &$0.0148$   &$0.0000$   &$0.0000$   &$0.0000$ \\
   &&$0.4060$   &$0.0867$   &$0.0032$   &$0.0001$   &$0.0000$ \\ \cline{2-7}
 & \multirow{2}{*}{\sf R1}
    &$0.2502$   &$0.0569$   &$0.0037$   &$0.0000$   &$0.0000$ \\
   &&$0.5015$   &$0.1669$   &$0.0094$   &$0.0001$   &$0.0000$ \\ \hline
\end{tabular}
\end{tabular}

$\ $ \\%
\centering \begin{tabular}{cc} 
Trace of CFD & Trace of inverse of CFD \\
 \small
\begin{tabular}{ccrrrrrr} \hline
& & \multicolumn{1}{c}{\!\!$\theta = 1.2$\!\!} & \multicolumn{1}{c}{$\theta = 2$} & \multicolumn{1}{c}{$\theta = 4$} & \multicolumn{1}{c}{$\theta = 8$} &\!\!\!\!\!\!  \\ \hline
 \multirow{8}{*}{\rotatebox{90}{$k = 1$}} &
 \multirow{2}{*}{\sf G}
    &$0.0000$   &$0.0000$   &$0.0000$   &$0.0000$ \\
   &&$0.0000$   &$0.0000$   &$0.0000$   &$0.0000$ \\  \cline{2-7}
 & \multirow{2}{*}{\sf G1}
    &$0.0602$   &$0.0000$   &$0.0000$   &$0.0000$ \\
   &&$0.0826$   &$0.0000$   &$0.0000$   &$0.0000$ \\ \cline{2-7}
 & \multirow{2}{*}{\sf R}
    &$0.0000$   &$0.0000$   &$0.0000$   &$0.0000$ \\
   &&$0.0000$   &$0.0000$   &$0.0000$   &$0.0000$ \\ \cline{2-7}
 & \multirow{2}{*}{\sf R1}
    &$0.0000$   &$0.0000$   &$0.0000$   &$0.0000$ \\
   &&$0.0000$   &$0.0000$   &$0.0000$   &$0.0000$ \\ \hline
 \multirow{8}{*}{\rotatebox{90}{$k = 5$}} &
 \multirow{2}{*}{\sf G}
    &$0.0000$   &$0.0000$   &$0.0000$   &$0.0000$ \\
   &&$0.0000$   &$0.0000$   &$0.0000$   &$0.0000$ \\ \cline{2-7}
 & \multirow{2}{*}{\sf G1}
    &$0.0008$   &$0.0000$   &$0.0000$   &$0.0000$ \\
   &&$0.0006$   &$0.0000$   &$0.0000$   &$0.0000$ \\ \cline{2-7}
 & \multirow{2}{*}{\sf R}
    &$0.0000$   &$0.0000$   &$0.0000$   &$0.0000$ \\
   &&$0.0000$   &$0.0000$   &$0.0000$   &$0.0000$ \\ \cline{2-7}
 & \multirow{2}{*}{\sf R1}
    &$0.0000$   &$0.0000$   &$0.0000$   &$0.0000$ \\
   &&$0.0000$   &$0.0000$   &$0.0000$   &$0.0000$ \\ \hline
 \multirow{8}{*}{\rotatebox{90}{$k = 10$}} &
 \multirow{2}{*}{\sf G}
    &$0.0000$   &$0.0000$   &$0.0000$   &$0.0000$ \\
   &&$0.0000$   &$0.0000$   &$0.0000$   &$0.0000$ \\ \cline{2-7}
 & \multirow{2}{*}{\sf G1}
    &$0.0000$   &$0.0000$   &$0.0000$   &$0.0000$ \\
   &&$0.0000$   &$0.0000$   &$0.0000$   &$0.0000$ \\ \cline{2-7}
 & \multirow{2}{*}{\sf R}
    &$0.0000$   &$0.0000$   &$0.0000$   &$0.0000$ \\
   &&$0.0000$   &$0.0000$   &$0.0000$   &$0.0000$ \\ \cline{2-7}
 & \multirow{2}{*}{\sf R1}
    &$0.0000$   &$0.0000$   &$0.0000$   &$0.0000$ \\
   &&$0.0000$   &$0.0000$   &$0.0000$   &$0.0000$ \\ \hline
\end{tabular}
\!\!&\!\!  \small 
\begin{tabular}{ccrrrrrr} \hline
& & \multicolumn{1}{c}{\!\!$\theta = 1.2$\!\!} & \multicolumn{1}{c}{$\theta = 2$} & \multicolumn{1}{c}{$\theta = 4$} & \multicolumn{1}{c}{$\theta = 8$} &\!\!\!\!\!\! \\ \hline
 \multirow{8}{*}{\rotatebox{90}{$k = 1$}} &
 \multirow{2}{*}{\sf G}
    &$0.0666$   &$0.0003$   &$0.0000$   &$0.0000$ \\
   &&$0.0061$   &$0.0000$   &$0.0000$   &$0.0000$ \\  \cline{2-7}
 & \multirow{2}{*}{\sf G1}
    &$0.1306$   &$0.0147$   &$0.0008$   &$0.0001$ \\
   &&$0.2645$   &$0.0000$   &$0.0000$   &$0.0000$ \\ \cline{2-7}
 & \multirow{2}{*}{\sf R}
    &$0.0668$   &$0.0001$   &$0.0000$   &$0.0000$ \\
   &&$0.0047$   &$0.0000$   &$0.0000$   &$0.0000$ \\ \cline{2-7}
 & \multirow{2}{*}{\sf R1}
    &$0.1109$   &$0.0126$   &$0.0007$   &$0.0000$ \\
   &&$0.1596$   &$0.0000$   &$0.0000$   &$0.0000$ \\ \hline
 \multirow{8}{*}{\rotatebox{90}{$k = 5$}} &
 \multirow{2}{*}{\sf G}
    &$0.0040$   &$0.0000$   &$0.0000$   &$0.0000$ \\
   &&$0.0000$   &$0.0000$   &$0.0000$   &$0.0000$ \\ \cline{2-7}
 & \multirow{2}{*}{\sf G1}
    &$0.0732$   &$0.0007$   &$0.0000$   &$0.0000$ \\
   &&$0.0508$   &$0.0000$   &$0.0000$   &$0.0000$ \\ \cline{2-7}
 & \multirow{2}{*}{\sf R}
    &$0.0049$   &$0.0000$   &$0.0000$   &$0.0000$ \\
   &&$0.0000$   &$0.0000$   &$0.0000$   &$0.0000$ \\ \cline{2-7}
 & \multirow{2}{*}{\sf R1}
    &$0.0618$   &$0.0007$   &$0.0000$   &$0.0000$ \\
   &&$0.0052$   &$0.0000$   &$0.0000$   &$0.0000$ \\ \hline
 \multirow{8}{*}{\rotatebox{90}{$k = 10$}} &
 \multirow{2}{*}{\sf G}
    &$0.0002$   &$0.0000$   &$0.0000$   &$0.0000$ \\
   &&$0.0000$   &$0.0000$   &$0.0000$   &$0.0000$ \\ \cline{2-7}
 & \multirow{2}{*}{\sf G1}
    &$0.0344$   &$0.0001$   &$0.0000$   &$0.0000$ \\
   &&$0.0102$   &$0.0000$   &$0.0000$   &$0.0000$ \\ \cline{2-7}
 & \multirow{2}{*}{\sf R}
    &$0.0001$   &$0.0000$   &$0.0000$   &$0.0000$ \\
   &&$0.0000$   &$0.0000$   &$0.0000$   &$0.0000$ \\ \cline{2-7}
 & \multirow{2}{*}{\sf R1}
    &$0.0302$   &$0.0000$   &$0.0000$   &$0.0000$ \\
   &&$0.0001$   &$0.0000$   &$0.0000$   &$0.0000$ \\ \hline
\end{tabular}
\end{tabular}

$\ $ \\%
\centering \begin{tabular}{cc} 
Trace of inverse of Laplace & Frobenius norm of inverse of Convdiff \\
 \small
\begin{tabular}{ccrrrrrr} \hline
& & \multicolumn{1}{c}{\!\!$\theta = 1.2$\!\!} & \multicolumn{1}{c}{$\theta = 2$} & \multicolumn{1}{c}{$\theta = 4$} & \multicolumn{1}{c}{$\theta = 8$} &\!\!\!\!\!\! \\ \hline
 \multirow{8}{*}{\rotatebox{90}{$k = 1$}} &
 \multirow{2}{*}{\sf G}
    &$0.0996$   &$0.0006$   &$0.0000$   &$0.0000$    \\
   &&$0.0801$   &$0.0000$   &$0.0000$   &$0.0000$ \\  \cline{2-7}
 & \multirow{2}{*}{\sf G1}
    &$0.2521$   &$0.0488$   &$0.0030$   &$0.0000$  \\
   &&$0.4555$   &$0.0919$   &$0.0011$   &$0.0001$   \\ \cline{2-7}
 & \multirow{2}{*}{\sf R}
    &$0.0975$   &$0.0005$   &$0.0000$   &$0.0000$    \\
   &&$0.0718$   &$0.0000$   &$0.0000$   &$0.0000$    \\ \cline{2-7}
 & \multirow{2}{*}{\sf R1}
    &$0.2196$   &$0.0346$   &$0.0017$   &$0.0001$    \\
   &&$0.4267$   &$0.0145$   &$0.0000$   &$0.0000$    \\ \hline
 \multirow{8}{*}{\rotatebox{90}{$k = 5$}} &
 \multirow{2}{*}{\sf G}
    &$0.0096$   &$0.0000$   &$0.0000$   &$0.0000$    \\
   &&$0.0004$   &$0.0000$   &$0.0000$   &$0.0000$    \\ \cline{2-7}
 & \multirow{2}{*}{\sf G1}
    &$0.1764$   &$0.0042$   &$0.0000$   &$0.0000$    \\
   &&$0.2590$   &$0.0002$   &$0.0000$   &$0.0000$    \\ \cline{2-7}
 & \multirow{2}{*}{\sf R}
    &$0.0083$   &$0.0000$   &$0.0000$   &$0.0000$    \\
   &&$0.0001$   &$0.0000$   &$0.0000$   &$0.0000$    \\ \cline{2-7}
 & \multirow{2}{*}{\sf R1}
    &$0.1483$   &$0.0021$   &$0.0000$   &$0.0000$    \\
   &&$0.1857$   &$0.0000$   &$0.0000$   &$0.0000$    \\ \hline
 \multirow{8}{*}{\rotatebox{90}{$k = 10$}} &
 \multirow{2}{*}{\sf G}
    &$0.0006$   &$0.0000$   &$0.0000$   &$0.0000$    \\
   &&$0.0000$   &$0.0000$   &$0.0000$   &$0.0000$    \\ \cline{2-7}
 & \multirow{2}{*}{\sf G1}
    &$0.1362$   &$0.0003$   &$0.0000$   &$0.0000$    \\
   &&$0.1464$   &$0.0000$   &$0.0000$   &$0.0000$    \\ \cline{2-7}
 & \multirow{2}{*}{\sf R}
    &$0.0005$   &$0.0000$   &$0.0000$   &$0.0000$    \\
   &&$0.0000$   &$0.0000$   &$0.0000$   &$0.0000$    \\ \cline{2-7}
 & \multirow{2}{*}{\sf R1}
    &$0.0970$   &$0.0000$   &$0.0000$   &$0.0000$    \\
   &&$0.0865$   &$0.0000$   &$0.0000$   &$0.0000$    \\ \hline
\end{tabular}
\!\!&\!\!  \small 
\begin{tabular}{ccrrrrrr} \hline
& & \multicolumn{1}{c}{\!\!$\theta = 1.2$\!\!} & \multicolumn{1}{c}{$\theta = 2$} & \multicolumn{1}{c}{$\theta = 4$} & \multicolumn{1}{c}{$\theta = 8$} & \multicolumn{1}{c}{$\theta = 30$} \\ \hline
 \multirow{8}{*}{\rotatebox{90}{$k = 1$}} &
 \multirow{2}{*}{\sf G}
    &$0.2679$   &$0.1042$   &$0.0144$   &$0.0006$   &$0.0000$ \\
   &&$0.5677$   &$0.2909$   &$0.0343$   &$0.0002$   &$0.0000$ \\  \cline{2-7}
 & \multirow{2}{*}{\sf G1}
    &$0.2218$   &$0.1232$   &$0.0456$   &$0.0121$   &$0.0001$ \\
   &&$0.6945$   &$0.5452$   &$0.3344$   &$0.1575$   &$0.0093$ \\ \cline{2-7}
 & \multirow{2}{*}{\sf R}
    &$0.2666$   &$0.1034$   &$0.0136$   &$0.0004$   &$0.0000$ \\
   &&$0.5697$   &$0.2954$   &$0.0390$   &$0.0003$   &$0.0000$ \\ \cline{2-7}
 & \multirow{2}{*}{\sf R1}
    &$0.2165$   &$0.1200$   &$0.0443$   &$0.0115$   &$0.0004$ \\
   &&$0.6901$   &$0.5364$   &$0.3246$   &$0.1453$   &$0.0062$ \\ \hline
 \multirow{8}{*}{\rotatebox{90}{$k = 5$}} &
 \multirow{2}{*}{\sf G}
    &$0.2584$   &$0.0188$   &$0.0000$   &$0.0000$   &$0.0000$ \\
   &&$0.3684$   &$0.0330$   &$0.0000$   &$0.0000$   &$0.0000$ \\ \cline{2-7}
 & \multirow{2}{*}{\sf G1}
    &$0.2665$   &$0.0919$   &$0.0125$   &$0.0004$   &$0.0000$ \\
   &&$0.5517$   &$0.2783$   &$0.0455$   &$0.0017$   &$0.0000$ \\ \cline{2-7}
 & \multirow{2}{*}{\sf R}
    &$0.2593$   &$0.0182$   &$0.0000$   &$0.0000$   &$0.0000$ \\
   &&$0.3731$   &$0.0372$   &$0.0000$   &$0.0000$   &$0.0000$ \\ \cline{2-7}
 & \multirow{2}{*}{\sf R1}
    &$0.2728$   &$0.0909$   &$0.0120$   &$0.0007$   &$0.0000$ \\
   &&$0.5371$   &$0.2570$   &$0.0406$   &$0.0010$   &$0.0000$ \\ \hline
 \multirow{8}{*}{\rotatebox{90}{$k = 10$}} &
 \multirow{2}{*}{\sf G}
    &$0.2198$   &$0.0031$   &$0.0000$   &$0.0000$   &$0.0000$ \\
   &&$0.2696$   &$0.0023$   &$0.0000$   &$0.0000$   &$0.0000$ \\ \cline{2-7}
 & \multirow{2}{*}{\sf G1}
    &$0.2763$   &$0.0688$   &$0.0039$   &$0.0001$   &$0.0000$ \\
   &&$0.4686$   &$0.1431$   &$0.0066$   &$0.0000$   &$0.0000$ \\ \cline{2-7}
 & \multirow{2}{*}{\sf R}
    &$0.2268$   &$0.0033$   &$0.0000$   &$0.0000$   &$0.0000$ \\
   &&$0.2641$   &$0.0037$   &$0.0000$   &$0.0000$   &$0.0000$ \\ \cline{2-7}
 & \multirow{2}{*}{\sf R1}
    &$0.2752$   &$0.0634$   &$0.0036$   &$0.0000$   &$0.0000$ \\
   &&$0.4618$   &$0.1340$   &$0.0033$   &$0.0000$   &$0.0000$ \\ \hline
\end{tabular}
\end{tabular}


\end{document}